\documentclass[reqno,11pt]{article}
\usepackage{times}
\usepackage[latin1]{inputenc}
\usepackage[T1]{fontenc}
\usepackage{amssymb}
\usepackage{amsmath}
\usepackage{amsthm}
\usepackage{latexsym}
\usepackage[dvips]{graphicx}
\usepackage[margin=3.5cm]{geometry}
\usepackage{mathrsfs}

\begin{document}

\def\C{{\mathbb C}}
\def\Z{{\mathbb Z}}
\def\N{{\mathcal N}}
\def\R{{\mathbb R}}
\def\T{{\mathbb T}}
\def\e{{\varepsilon}}
\def\g{{\gamma}}
\def\G{{\Gamma}}
\def\s{{\sigma}}
\def\r{{\rho}}
\def\o{{\omega}}
\def\a{{\alpha}}
\def\d{{\delta}}
\def\l{{\lambda}}
\def\k{{\kappa}}
\def\p{{\prime}}
\def\for{{\,\,\,\forall \,\,\, }}
\def\ex{{\,\,\, \exists \,\,}}
\def\W{ {\mathcal{C}_a } }
\def\E{ {\mathcal{E} } }
\def\B{ {\mathcal{B} } }
\def\A{ {\mathcal{A} } }
\def\K{ {\mathcal{K} } }
\def\H{ {\mathcal{H} } }
\def\S{ {\mathcal{S} } }
\def\Y{ {\mathcal{Y} } }
\def\D{{\mathcal{D}}}
\def\L{ {\mathcal{L} } }
\def\pr{\vskip3pt \nl{\bf Proof. \,}}
\def\rin{ r_{\infty}}
\def\Im{\,\, \mbox{Im} \,\, }
\def\Re{\,\, \mbox{Re} \,\, }
\def\mat{\,\,\mbox{mat}}
\def\meas{\,\mbox{meas}\,}
\def\supp{\,\mbox{supp}\,}
\def\id{\,\mbox{id} \,}
\def\diam{\,\mbox{\upshape diam} \,}
\def\Russ{R\"{u}{\ss}mann}
\def\+R{+_{_{ \!\! \R}}}
\def\ex{{\mbox{\scriptsize{ex}}}}
\def\M{{\mathcal{M}}}
\def\vs{{\varsigma}}
\def\ve{{\varepsilon}}
\def\diag{\,\mbox{diag}\,}
\def\tr{\,\mbox{tr}\,}
\def\minus{{\smallsetminus}}
\def\Ha{{\mathscr{H}}}
\def\dstar{{\star\star}}
\def\mD{{\mathbf D}}
\def\div{{\,\mbox{div}\,}}
\def\curl{{\,\mbox{curl}\,}}

\def\relation{{\rho_- \leq \e^{3/2}}}


\def\nl{\vglue0.3truemm\noindent}
\numberwithin{equation}{section}

\newtheorem{teo}{Theorem}[section]
\newtheorem{pro}[teo]{Proposition}
\newtheorem{lem}[teo]{Lemma}
\newtheorem{defin}[teo]{Definition}
\newtheorem{oss}[teo]{Remark}
\newtheorem{cor}[teo]{Corollary}
\renewcommand{\thefootnote}{\alph{footnote}}


\title{On the limit as the surface tension 
\\
and density ratio tend to zero
\\
for the two-phase Euler equations
}

\author{ 
Fabio Pusateri
\\
\vspace{-.2truecm}
{\footnotesize Courant Institute of Mathematical Sciences}
\\
\vspace{-.2truecm}
{\footnotesize New York University}
\\
\vspace{-.2truecm}
{\footnotesize 251 Mercer Street, New York, N.Y. 10012, USA}
\\
\vspace{-.2truecm}
\textit{\footnotesize{pusateri@cims.nyu.edu}}
}

\date{}

\maketitle

\begin{abstract}
We consider the free boundary motion of two perfect incompressible fluids
with different densities $\rho_+$ and $\rho_-$, separated by a surface of discontinuity along which 
the pressure experiences a jump proportional to the mean curvature by a factor $\e^2$.
Assuming the Raileigh-Taylor sign condition, and $\relation$,
we prove energy estimates uniform in $\rho_-$ and $\e$.
As a consequence, we obtain convergence of solutions of the interface problem
to solutions of the free boundary Euler equations in vacuum without surface tension as $\e , \rho_- \rightarrow 0$.
\end{abstract}


\section{Introduction}

\subsection{Description of the problem}
We consider the interface problem between two incompressible and inviscid fluids 
that occupy domains $\Omega^+_t$ and $\Omega^-_t$ in $\R^n$ ($n \geq 2$) at time $t$.
We assume $\Omega_0^+$ is compact and $\R^n = \Omega^+_t \cup \Omega^-_t \cup S_t$ where $S_t := \partial \Omega^\pm_t$.
We let $v_\pm$, $p_\pm$ and $\rho_\pm > 0$ denote respectively the velocity, 
the pressure and the constant density of the fluid occupying the region $\Omega^\pm_t$.
We assume the presence of surface tension on the interface, which is argued on physical basis 
to be proportional to the mean curvature $\k_+$ of the hypersurface $S_t$.

The equations of motion are given by\footnote{Here we are introducing the notation
$f = f_+ \chi_{\Omega_t^+} + f_- \chi_{\Omega_t^-}$ for any $f_\pm$ defined on $\Omega_t^\pm$.}
\begin{equation*}
\tag{E}
\label{E}
\left\{
\begin{array}{ll}
\rho (v_t + v \cdot \nabla v) = - \nabla p   \quad &   x \in \R^n \smallsetminus S_t
\\
\\
\nabla \cdot v = 0    &   x \in \R^n \smallsetminus S_t
\\
\\
v (0,x) = v^0 (x) 
	& x \in \R^n \smallsetminus S_0 \, ,
\end{array}
\right.
\end{equation*}
with corresponding boundary conditions for the interface evolution and pressure's jump given by
\begin{equation}
\tag{BC}
\label{BC}
\left\{
\begin{array}{l}
\partial_t + v_\pm \cdot \nabla  \,\, \mbox{is tangent to} \,\,
		\left\{ (t,x) \, | \, x \in S_t \right\} 
\\
\\
p_+ (t,x) - p_- (t,x) = \e^2 \k_+ (t,x) \,\, , \,\,\, x \in S_t  \, .
\end{array}
\right.
\end{equation}
We are interested in analyzing the asymptotic behaviour of solutions of the above equations when 
$\e , \rho_- \rightarrow 0$.
Our result, based on the previous works of Shatah and Zeng \cite{shatah1,shatah2,shatah3},
is convergence to the solution $(v^\infty, \partial \Omega_t^\infty)$ of the system
%
%
\begin{equation}
\label{E_0}
\tag{$\mbox{E}_0$}
\left\{
\begin{array}{ll}
\rho_+ ( \partial_t v^\infty + v^\infty \cdot \nabla v^\infty) = - \nabla p^\infty  \quad &   x \in \Omega^\infty_t
\\
\\
\nabla \cdot v^\infty = 0    &   x \in \Omega^\infty_t
\\
\\
v^\infty (0,x) = v^0_+ (x) 
        	& x \in \Omega^+_0 \, ,
\end{array}
\right.
\end{equation}
with corresponding boundary conditions
\begin{equation}
\label{BC_0}
\tag{$\mbox{BC}_0$}
\left\{
\begin{array}{l}
\partial_t  +  v^\infty \cdot \nabla  \,\, \mbox{is tangent to} \,\, 
\left\{ (t,x) \, | \, x \in S_t^\infty \right\}  
\\
\\
p^\infty (t,x) = 0  \,\, , \,\,\, x \in \partial \Omega_t^\infty   \,  .
\end{array}
\right.
\end{equation}

Equations \eqref{E_0}-\eqref{BC_0} typically model the free boundary motion of a drop 
of perfect incompressible fluid in vacuum (one-phase problem).
The system \eqref{E}-\eqref{BC} models instead 
the motion of two perfect fluids with different densities separated by an interface 
moving with the normal components of the velocities of the two fluids (two-phase problem).
When considering the one-phase problem one can think of 
a fluid with very small density $\rho_-$ (air, for instance) replacing vacuum. 
In this case, \eqref{E_0}-\eqref{BC_0} can still be considered as an idealized model
but, even when $\rho_-$ is very small compared to $\rho_+$, the two-phase system 
provides a more accurate description of the motion.
Similarly, for $\rho_- \ll \rho_+$ and $\e \ll 1$,  \eqref{E}-\eqref{BC} represent a more accurate model
for the problem of one fluid surrounded by air in the presence of small, but not negligible, 
surface tension effects holding the fluid together.

Due to their physical and mathematical interest, 
freeboundary problems for Euler equations have been extensively studied in recent years. 
Following the breakthrough of Wu in \cite{Wu1,Wu2}, where local well-posedness for arbitrary data in 
Sobolev spaces was proved in $2$ and $3$ dimensions for the irrotational gravity water wave problem,
a vast body of literature has been produced.
Many works have dealt with the water wave problem 
with or without surface tension and with or without vorticity,
see \cite{Lindblad,ChriLind,CoutShko2,shatah1,shatah3} and references therein.

A natural question related to the well-posedness of this set of problems
is the one concerning the relation between their solutions in regimes
which are a perturbation of one another.
For the one-phase problem \eqref{E_0} with vanishing surface tension
- i.e. where the boundary condition for the pressure \eqref{BC_0} is replaced by $p^\infty = \e \k^\infty$ -
it was proved in \cite{AmbroseMasmoudi1}, for the irrotational 2-d case,
and in \cite{shatah1}, for the general case,
that solutions to this problem converge to solution of \eqref{E_0}-\eqref{BC_0} as $\e \rightarrow 0$.
Recently, Cheng, Coutand and Shkoller \cite{CoutShko3} and the author \cite{VS} proved that
solutions of \eqref{E}-\eqref{BC} with $\e = 1$ converge to solutions 
of the one-phase problem with surface tension as $\rho_- \rightarrow 0$.

In absence of surface tension, i.e. $\e = 0$ in \eqref{BC}, 
the two-phase problem \eqref{E}-\eqref{BC} 
for the free boundary motion of two fluids is ill-posed
due to the Kelvin-Helmotz instability \cite{Ebin2}.
In \cite{Beale} it is shown how, indeed, the surface tension regularizes the linearized problem.
For the irrotational problem with surface tension, Ambrose \cite{Ambrose} 
and Ambrose and Masmoudi \cite{AmbroseMasmoudi2} proved well-posedness 
respectively in $2$ and $3$ dimensions.
Cheng, Coutand and Shkoller \cite{CoutShko1} proved well-posedness for the full 3-d problem with rotation.
Well-posedness is also obtained (in any dimension) by Shatah and Zeng \cite{shatah3}.

We recall that the free boundary problem for Euler equations in vacuum \eqref{E_0}-\eqref{BC_0}
is also known to be ill-posed \cite{Ebin1} due to Rayleigh-Taylor instability,
which occurs if one does not assume the sign condition
\begin{equation}
\label{RT}
\tag{RT}
- \nabla_{N} p^\infty (x,t) \geq a > 0 \hskip 10 pt \for \, x \in S_t \, .
\end{equation}

The result we are presenting here is largely based on the geometric intuition and techniques
introduced in \cite{shatah1} and further developed in \cite{shatah2,shatah3}.
Our paper is organized as follows. 
The geometric approach of \cite{shatah1,shatah2} is presented in section \ref{secgeometry} and
an explanation of the geometric intuition behind the Kelvin-Helmotz and Raileigh-Taylor instabilities
is given in \ref{seclinearization}.
In section \ref{secEE} we define the energy for the problem and state
theorems on energy estimates which are independent of $\e$ and $\rho_-$.
As a corollary, we state the result about convergence 
of solutions of \eqref{E}-\eqref{BC} to solutions of \eqref{E_0}-\eqref{BC_0}.
Section \ref{secproof} is dedicated to the proofs of the statements.
In \ref{secpreliminaries} we first collect some preliminary estimates
and then derive an evolution equation for the mean-curvature $\k_+$ (lemma \ref{lemD_t^2}),
upon which our energy is based.
In \ref{secproofteoVS1} we prove that our energy controls in a suitable
fashion the Sobolev norms of the velocity fields and the mean-curvature of the free surface.
In \ref{secproofteoEE} we study the time-evolution of the energy,
where an extra higher order energy term (due to the Kelvin-Helmotz instability) will appear.
Assuming some smallness condition on $\rho_-$ as a function of $\e$, 
the extra energy term is controlled in \ref{secenergyinequality},
therefore concluding the proof of energy estimate.
In the appendix 
we gathered some technical material contained in \cite{shatah1,shatah2} used in our proofs.

\subsection{The geometric approach to Euler equations} 
\label{secgeometry}
It is well-known that the interface problem between two fluids
has a variational formulation on a subspace of volume-preserving homeomorphisms.
For the water wave problem, this was observed for the first time by Arnold in his seminal paper \cite{Arnold66},
where he pointed out that Euler equations for the motion 
of an inviscid incompressible fluid can be viewed as the geodesic flow 
on the infinite-dimensional manifold of volume-preserving diffeomorphisms.
This point of view has been adopted by several authors in works 
such as \cite{Brenier,EbinMarsden,Shnirelman},
and more recently by Shatah and Zeng in \cite{shatah1,shatah2,shatah3}.

\subsubsection{Lagrangian formulation}
We first recall that \eqref{E}-\eqref{BC} has a conserved energy\footnote{
Notice that the conserved energy does not control 
the $L^2$ norm of $v_-$ in the asymptotic regime $\rho_- \rightarrow 0$.
}
\begin{equation}
\label{energy0}
E = E_0 (S_t, v) = \int_{\R^n \minus S_t} \frac{\rho {|v|}^2 }{2} \, dx + \e^2 \int_{S_t} \, dS =: 
										\int_{\R^n \minus S_t} \frac{\rho {|v|}^2 }{2} \, dx +  \e^2 S (S_t) \, .
\end{equation}
For $y \in \Omega_0^\pm$ we define $u_\pm (t,y)$ to be the Lagrangian coordinate map associated to
the velocity field $v_\pm$, i.e the solution of the ODE
\begin{equation}
\label{lagrangianmap}
\frac{d x}{dt} = v_\pm (t, x) \hskip8pt, \hskip10pt  x(0,y)= y \hskip 10pt \for y \in \Omega_0^\pm \, .
\end{equation}
Also, for any vector field $w$ on $\R^n \minus S_t$ we define its material derivative by
\begin{equation*}
\mathbf{D}_t w := w_t + v \cdot \nabla w =  {(w \circ u)}_t \circ u^{-1} \, .
\end{equation*}
In \cite[sec. 2]{shatah2} the authors derive from (\ref{E})-(\ref{BC}) 
an equation for the physical pressure: 
\begin{equation}
\label{physicalp}
\left\{
\begin{array}{lcl} 
- \Delta p & = & \rho \tr (Dv^2)  
\\
\\
\left. p_\pm \right|_{S_t}  &  =  &
						\mathcal{N}^{-1} \left\{ - \frac{1}{\rho_\mp} \mathcal{N}_\mp  \e^2 \k_\mp
						- 2 \nabla_{v_+^\top - v_-^\top} v_+^\bot
				    - \Pi_+ ( v_+^\top, v_+^\top ) - \Pi_- ( v_-^\top, v_-^\top ) \right.
				    \\
				    &  &  \left. 
				    - \nabla_{N_+} \Delta_+^{-1}  \tr (Dv^2)  - \nabla_{N_-} \Delta_-^{-1}  \tr (Dv^2)
				    \right\}
\end{array}
\right.
\end{equation}
where $\Pi_\pm$ denotes the second fundamental form of the hypersurface $S_t$ (with respect to the outward unit normal vector $N_\pm$ relative to the domain $\Omega_t^\pm$) and $\mathcal{N}$ is given by
\begin{equation}
\label{N}
\mathcal{N} := \frac{\mathcal{N}_+}{\rho_+}  +   \frac{\mathcal{N}_-}{\rho_-} \, ,
\end{equation}
with $\mathcal{N}_\pm$ denoting the Dirichlet-to-Neumann operator on the domain $\Omega^\pm_t$.
From (\ref{lagrangianmap}) we see that in Lagrangian coordinates Euler equations  assume the form
\begin{equation}
\label{Elagrangian}
\rho u_{tt} = - \nabla p \circ u \hskip 20pt u(0) = \id_{\Omega_0}
\end{equation}
with $p$ determined by (\ref{physicalp}).

Since $v$ is divergence free, $u_\pm$ are volume-preserving maps on $\Omega_0^\pm$.
Moreover, $u_+(t, S_0) = u_- (t, S_0)$
even if the restriction of $u_+$ and $u_-$ to $S_0$ do not coincide in general.
This leads to the definition of the space $\G$ of admissible Lagrangian maps for the interface problem:
\begin{align}
\nonumber
\G = & \left\{ \Phi = \Phi_+ \chi_{\Omega_0^+} + \Phi_- \chi_{\Omega_0^-} \,\, \mbox{s.t.} \,\,\,
										\Phi_\pm : \Omega_0^\pm \rightarrow \Phi_\pm (\Omega_0^\pm) \,\, 
										\right. 
										\\
										& \left.	
										\mbox{is volume-preserving homeomorphism}
										\, , \,
										\partial \Phi_\pm (\Omega_0^\pm) = \Phi_\pm (\partial \Omega_0^\pm)  \right\} \, .
\label{defG}
\end{align}
Denoting $S (\Phi) = \int_{\Phi(S_0)} dS$,
we can rewrite the energy (\ref{energy0}) in Lagrangian coordinates as
\begin{equation*}
E_0 (u, u_t) = \int_{\R^n \minus S_0} \frac{ \bar{\rho} {|u_t|}^2 }{2} \, dy + \e^2 S(u)
\end{equation*}
where $ (u,u_t) $ is in the tangent bundle of $\G$ and $\bar{\rho} := \rho \circ u$.
The conservation of the above energy suggests that \eqref{E}-\eqref{BC}
has a Lagrangian action
\begin{equation}
\label{action}
I (u) = \int \int_{\R^n \minus S_0} \frac{\bar{\rho} {|u_t|}^2 }{2} \, dy \, dt \, - \e^2 \int S(u) \, dt \, .
\end{equation}

\subsubsection{The geometry of $\G$}
\label{secgeometryG}
In order to derive the Euler-Lagrange equations associated to the action $I$,
we consider $\G$ as a submanifold of $L^2 ( \bar{\rho} dy )$
and identify its tangent and normal spaces.
It is easy to see that the tangent space of $\G$ at the point $\Phi$ is given by divergence-free vector fields
with matching normal components in Eulerian coordinates\footnote{We follow the 
convention used in \cite{shatah2} where the Lagrangian description of any vector field 
$X : \Phi (\Omega_0) \rightarrow \R^n$ is denoted by $\bar{X} = X \circ \Phi$.}
\begin{equation}
\nonumber
T_{\Phi} \G = \left\{ \bar{w} : \R^n \minus S_0 \rightarrow \R^n \,\, :
					\nabla \cdot w = 0 \,\, 
					\mbox{ and } \left. w_+^\bot + w_-^\bot \right|_{\Phi(S_0)} = 0 
					\right\} \, ,
\label{defTG}
\end{equation}
while the normal space is
\begin{equation}
\label{TPhiort}
{(T_\Phi \G)}^\bot = \left\{ - (\nabla \psi) \circ \Phi \, : \,  
								\rho_+ \psi_+ \left|_{\Phi(S_0)} \, =  \, \rho_- \psi_- \right|_{\Phi (S_0)} =: \psi^S  \right\} \, .
\end{equation}
A critical path $u(t, \cdot)$ of $I$ satisfies
\begin{equation}
\label{critpath}
\bar{\mathscr{D}}_t u_t + \e^2 S^\p (u) = 0
\end{equation}
where $S^\p (u)$ denotes the tangential gradient of $S(u)$ and 
$\bar{\mathscr{D}}_t$ is the covariant derivative on $\G$ along $u(t)$.
In order to verify that the Lagrangian map associated to a solution of \eqref{E}-\eqref{BC}
is indeed a critical path of (\ref{action}) one needs to compute $S^\p$ and $\bar{\mathscr{D}}_t$.
Let
\begin{equation}
u_{tt} = \bar{\mathscr{D}}_t  u_t  +  II_{u(t)} (\bar{v}, \bar{v})
\end{equation}
where $II_{u(t)} (\bar{w}, \bar{v}) \in {(T_{u(t)} \G)}^\bot$ denotes the second fundamental form on $T_{u(t)} \G$.
From \eqref{TPhiort} there exists a unique scalar function $p_{v,v}$ defined on $\R^n \minus S_t$ such that 
\begin{equation*}
II_{u(t)} (\bar{v}, \bar{v}) = - \nabla p_{v,v} \circ u \in {\left( T_{u(t)} \G \right)}^\bot
\end{equation*}
In \cite{shatah2} it is shown that $p_{v,v}$ is given by
\begin{equation}
\label{p_vw}
\left\{
\begin{array}{rcl} 
- \Delta p_{v,v} & =  & \tr {(Dv)}^2
\\
\\
\left. p^\pm_{v,v} \right|_{S_t} &  = & \frac{1}{\rho_\pm} p^S_{v,v} =  
						- \frac{1}{\rho_\pm} \mathcal{N}^{-1} \left\{ 2 \nabla_{v_+^\top - v_-^\top} v_+^\bot
				    - \Pi_+ ( v_+^\top, v_+^\top ) - \Pi_- ( v_-^\top, v_-^\top ) \right.
				    \\
				    &  &  \left. 
				    - \nabla_{N_+} \Delta_+^{-1}  \tr {(Dv)}^2  - \nabla_{N_-} \Delta_-^{-1}  \tr {(Dv)}^2
				    \right\}
				    =:  - \frac{1}{\rho_\pm} \mathcal{N}^{-1} a
				    \, .
\end{array}
\right. 
\end{equation}
Hence, in Eulerian coordinates we can write
\begin{equation}
\label{covariant}
\mathscr{D}_t v  :=  \left(  \bar{\mathscr{D}}_t \bar{v}  \right) \circ u^{-1} = \mathbf{D}_t v + \nabla p_{v,v}  \, .
\end{equation}
We point out that for the water wave problem \eqref{E_0}-\eqref{BC_0} 
the second fundamental form on the space of admissible Lagrangian maps 
has a simpler expression, namely
$$II_{u(t)}^\star (\bar{v}, \bar{v}) = - \nabla p^\star_{v,v} \circ u $$ 
with
\begin{equation}
\label{pwaterwave}
\left\{
\begin{array}{lll}
- \Delta p^\star_{v,v}  & = &   \tr {(Dv)}^2    
\\
\\
\left. p^\star_{v,v} \right|_{\partial \Omega_t} & = & 0  \, .
\end{array}
\right.
\end{equation}
Observe that $p^\star_{v,v}$ coincides with $p^\infty$ in equation \eqref{E_0}-\eqref{BC_0}.

To compute $S^\p (u)$ one observes that for any $\bar{w} \in T_u \G$ 
the formula for the variation of surface area gives
\begin{equation*}
{ \langle S^\p (u) , \bar{w} \rangle }_{L^2 (\R^n \minus S_0, \rho dy)} = \int_{S_t} \k_+ w_+^\bot \, dS \, .
\end{equation*}
Then it is not hard to verify that the unique representation in Eulerian coordinates
of $S^\p (u)$ as a functional acting on $T_u \G$  is
\begin{equation}
\label{Sp}
S^\p (u) = \nabla p_\k  \hskip 15pt  \mbox{with}   \hskip 15pt  
					p_\k^\pm = \frac{1}{\rho_- \rho_+} \H_\pm \mathcal{N}^{-1} \mathcal{N}_\mp \k_\mp \, ,
\end{equation}
where $\H_\pm$ denotes the harmonic extension in the domain $\Omega_t^\pm$.
From \eqref{physicalp}, \eqref{p_vw} and \eqref{Sp}
one obtains the identity $p = \rho ( p_{v,v}  +  \e^2 p_\k )$,
and we see from  \eqref{covariant} and \eqref{Sp}
that a solution of \eqref{critpath} equivalently satisfies
\begin{equation}
\label{Ematerial}
\mathbf{D}_t v  + \nabla p_{v,v} + \e^2 \nabla p_\k  = 0 \, ,
\end{equation} 
which is exactly \eqref{Elagrangian} in Eulerian coordinates.

\subsubsection{Linearized equation and instability for water waves problems}
\label{seclinearization}

The Lagrangian formulation discussed above provides a convenient setting to study the linearization of the problem. 
Considering variations around the solution $u_t$ of \eqref{critpath} and taking 
a covariant derivative with respect to the variation parameter,
one obtains the following linearization for $\bar{w} (t, \cdot) \in T_{u(t)} \G$:
\begin{equation}
\label{linearization}
\bar{\mathscr{D}}^2_t \bar{w} + \bar{\mathscr{R}} (u) (\bar{u_t}, \bar{w}) u_t  +  \e^2 \bar{\mathscr{D}}^2 S (u) \bar{w} = 0 \, ,
\end{equation}
where $\bar{\mathscr{R}}$ denotes the curvature tensor of the manifold $\G$
and $\bar{\mathscr{D}}^2 S (u)$ is the projection on $T_{u} \G$ of the second variation of the surface area.
Both of these linear operators acting on $T_u \G$ 
play a central role in the understanding of the problem and in the derivation of 
high-order energies based upon their leading order terms.
In \cite{shatah1} a general formula for $\bar{\mathscr{D}}^2 S (u)$ is derived.
For the interface problem its leading order term $\bar{\mathscr{A}}$ 
is given in Eulerian coordinates by \cite[pp. 857-858]{shatah2}
\begin{eqnarray}
\nonumber
\mathscr{A}(u) (w) & = & \nabla f_+ \chi_{\Omega^+}   +    \nabla f_- \chi_{\Omega^-}
\\
\mbox{with} \qquad
f_\pm & = &  \frac{1}{\rho_+ \rho_-} \H_\pm \mathcal{N}^{-1} \mathcal{N}_\mp (- \Delta_{S_t}) w_\pm^\bot  \, .
\label{A}
\end{eqnarray}
It is easy to see that $\bar{\mathscr{A}}$ is a third-order\footnote{
Assuming  $S_t$ is smooth enough.}
self-adjoint and positive semi-definite operator with
$$ \bar{\mathscr{A}} (u) ( \bar{w}, \bar{w}) = | \nabla w_\pm^\bot |_{L^2 (S_t)}^2 \, .$$
Further computations \cite[pp 859 - 860]{shatah2}
show that the leading-order 
term $\bar{\mathscr{R}}_0 (u) (\bar{v})$ of the unbounded sectional curvature operator 
$\bar{\mathscr{R}}(u) (\bar{v}, \cdot) \bar{v}$
is given in Eulerian coordinates by
\begin{eqnarray*}
\mathscr{R}_0 (u) (\bar{v}) w & = & \nabla f_+ \chi_{\Omega^+}   +    \nabla f_- \chi_{\Omega^-}
\\
\mbox{with} \qquad 
f_\pm & = & \frac{1}{\rho_+ \rho_-} \H_\pm \mathcal{N}^{-1} \mathcal{N}_\mp  \nabla_{v_+^\top - v_-^\top}
\mathcal{N}^{-1} \mathcal{D} \cdot \left( w_\pm^\bot (v_+^\top - v_-^\top) \right) \, .
\end{eqnarray*}
%
%
Noticing that $\bar{\mathscr{R}}_0 (u)$ is a second-order negative semi-definite differential operator,
we immediately see that the linearized Euler equations would be ill-posed for $\e = 0$.
This is the so-called {\it Kelvin-Helmotz instability} for the two fluids interface problem,
occuring in the absence of surface tension.

We mention that the same geometric setting described above has been initially
developped by Shatah and Zeng in \cite{shatah1}, where they treated the problem of 
a priori energy estimated for Euler equations in vacuum.
In \cite[sec 2.2]{shatah1} the authors showed 
that the differential operators involved in the linearization \eqref{linearization} satisfy
\begin{eqnarray*}
\bar{\mathscr{R}} (\bar{v}, \bar{w}) & = & \bar{\mathscr{R}}_0^\star (u) + \, \mbox{bounded operators} 
\\
\bar{\mathscr{D}}^2 S (u) & = & \bar{\mathscr{A}}^\star (u) + \, \mbox{second-order differential operators} 
\end{eqnarray*}
with
\begin{eqnarray*}
\bar{\mathscr{R}}_0^\star (u) \bar{w} \cdot \bar{w}
							= \int_{S_t}  - \nabla_{N} p^\star_{v,v} {\left| \nabla w^\bot \right|}^2 \, dS
\hskip10pt , \hskip20pt
\bar{\mathscr{A}}^\star (u)  \bar{w} \cdot \bar{w}  
							= \int_{S_t} {\left| \nabla w^\bot \right|}^2 \, dS \, .
\end{eqnarray*}
Since also in this case $\bar{\mathscr{A}}^\star (u)$ is generated by the presence of surface-tension,
we see that \eqref{linearization} is ill-posed for $\e=0$ if one does not assume
the sign condition \eqref{RT}. This is the so called {\it Raileigh-Taylor} instability for the water wave problem.

\section{Theorems on Energy Estimates}
\label{secEE}
Following \cite{shatah1,shatah2} we define a set of neighbouring hypersurfaces of the initial hypersurface $S_0$.

\begin{defin}
\label{defLambda}
Let $\Lambda = \Lambda (S, s, \delta, L)$ for some $s > \frac{n+1}{2}$, $L, \delta > 0$ 
be the collection of all hypersurfaces $\tilde{S}$
such that $(a)$ there exists a diffeomorphism $F: S \rightarrow \tilde{S} \subset \R^n$ with
\begin{equation*}
{|F - \id_S|}_{H^s} (S) < \delta
\end{equation*}
and $(b)$ ${|\kappa|}_{ H^{s-2} (\tilde{S}) } < L$ for any $\tilde{S} \in \Lambda$.
Define $\Lambda_0 := \Lambda (S_0, 3k - \frac{1}{2}, \delta, L)$ for some $k$ satisfying $3k > \frac{n}{2} + 2$,
with $0 < \delta \ll 1$ and $L>0$ to be determined later.
\end{defin}

\nl
We now define the energy for \eqref{E}-\eqref{BC}.
\begin{defin}
Let $k$ be any integer such that $3 k > \frac{n}{2} + 2$.
Consider domains  $\Omega_t^\pm \subset \R^n$ with $\Omega_t^+$ compact 
and interface
$S_t = \partial \Omega_t^\pm \in \Lambda_0$.
Let $v (t, \cdot) \in H^{3k} (\R^n \minus S_t)$ be any divergence-free vector field 
with $v_+^\bot + v_-^\bot = 0$.
Let $\omega_\pm$ denote the $\curl$ of $v_\pm$, that is $\omega_i^j = \partial_i v^j - \partial_j v^i$,
and define
\begin{equation}
\label{barN}
\bar{\mathcal{N}} := 
\frac{1}{\rho_+\rho_-} \mathcal{N}_+ \mathcal{N}^{-1}  \mathcal{N}_-  \, .
\end{equation}
We define our energy by
\begin{equation}
\label{Energy}
E (S_t, v(t, \cdot) ) =   E_1 + E_2 + E_{RT}  + {| \o_+ |}^2_{ H^{3k-1} (\Omega_t^+) } 
+ \e {| \o_- |}^2_{ H^{3k-1} (\Omega_t^-) } 
\end{equation}
where
\begin{align}
\label{E_1}
E_1  & := \frac{1}{2} \int_{S_t} {\left| \bar{\N}^\frac{1}{2} { (- \Delta_{S_t} \bar{\N} )}^{k-1}  
		\mD_{t_+} \k_+ \right|}^2 \, dS 
		\\
		\nonumber
		& = \frac{1}{2}  \int_{S_t} \mathbf{D}_{t_+} \k_+
		\bar{\N} { (-  \Delta_{S_t} \bar{\N} )}^{2k-2} \mathbf{D}_{t_+} \k_+  \, dS  \, ,
\\
\nonumber
\\
\label{E_2}
E_2 & :=  \frac{\e^2}{2}  \int_{S_t}   {\left| \nabla^{\top} {( - \bar{\N} \Delta_{S_t} )}^{k - 1} \bar{\N} \k_+ \right|}^2  \, dS
             	\\
		\nonumber
 	    	& =
	     	-  \frac{\e^2}{2}  \int_{S_t} \k_+ \bar{\N} { (-  \Delta_{S_t} \bar{\N} )}^{2k-1} \k_+  \, dS   \, ,
\\
\nonumber
\\													
\label{E_RT}
E_{RT}  & := \frac{\rho_+ + \rho_-}{2}  \int_{S_t}  
		- \nabla_{N_+} p^\star_{v,v} {\left| {(- \bar{\N} \Delta_{S_t} )}^{k-1} \bar{\N} \k_+ \right|}^2 \, dS \, .
\end{align}
\end{defin}

\nl
The following proposition establishes bounds
of relevant Sobolev norms of the velocity fields and mean-curvature in terms of the energy.

\begin{pro}
\label{teoVS1}
Let $3k > \frac{n}{2} + 2$ and assume \eqref{RT}.
Then, for $S_t \in \Lambda_0$, there exists a uniform constant $C_0$ such that
\begin{align}
\label{ekappaenergy}
&  {| \kappa_+ |}^2_{ H^{3k-2}(S_t) } \hskip5pt , \hskip 7pt \e^2 {| \kappa_+ |}^2_{ H^{3k-1}(S_t) } \leq C_0 (1 + E)
\\
\label{venergy+}
&  {| v_+ |}^2_{ H^{3k} (\Omega_t^+) }  \leq  C_0 (1 + E + E_0)
\\
\label{venergy-low}
&  {| v_- |}^2_{ H^{3k - 1} (\Omega_t^-) }  \leq  C_0 (1 + E + E_0) 
\\
\label{venergy-}
&  \e {| v_- |}^2_{ H^{3k} (\Omega_t^-) }  \leq  C_0 {(1 + E + E_0)}^2  \, .
\end{align}
\end{pro}
\nl
Using the above proposition we will prove
\begin{teo}[\bf Energy Estimates]
\label{teoEE}
Let $3k > \frac{n}{2} + 2$ and initial data\footnote{
The regularity of hypersurfaces in $\R^n$ is intended in the sense of local coordinates:
an hypersurface is $H^s$ for $s > \frac{n}{2} + 1$ if it can be locally represented
as the graph of $H^s$-functions.
} $S_0 \in H^{3k}$ and $v_0 \in H^{3k} (\Omega_0)$ be given.
Denote by
\begin{equation*}
S_t \in H^{3k} \hskip8pt \mbox{and} 
   \hskip8pt v(t, \cdot) \in C \left( H^{3k} (\R^n \smallsetminus S_t) \right) \, ,
\end{equation*}
the corresponding solution of \eqref{E}-\eqref{BC}.
Then, there exists $L > 0$ and a time $t^\star > 0$, depending only on 
${|v(0, \cdot)|}_{H^{3k} (\R^n \smallsetminus S_t)}$, $\Lambda_0$ and $L$,
such that $S_t \in \Lambda_0$ and ${| \kappa |}_{H^{3k - 5/2} (S_t) } \leq L$ for all $0 \leq t \leq t^\star$.
Moreover, assuming the Raileigh-Taylor sign condition \eqref{RT} and
\begin{equation}
\label{condition}
\relation \, ,
\end{equation}
the following energy estimate holds for $0 \leq t \leq t^\star$:
\begin{equation}
\label{energyestimate}
E ( S_t, v(t, \cdot) )  \leq 3 E ( S_0, v(0, \cdot) ) + C_1 + \int_0^t P (E_0, E(S_{t^\p}, v(t^\p, \cdot) )) \, dt^\p
\end{equation}
where $P$ is a polynomial with positive coefficients determined only by the set $\Lambda_0$,
and the constant $C_1$  depends only on $\Lambda_0$
and the $H^{3k - \frac{3}{2}} (\R^n \smallsetminus S_0)$-norm of $v_0$.
In particular, there exists a small time $T^\infty > 0$ 
and a constant $C_0$, depending only on the initial data and the set $\Lambda_0$, such that 
\begin{equation}
\label{unifEbound}
\sup_{t \in [0,T^\infty]} E ( S_t, v(t, \cdot) )  \leq C_0  \,  .
\end{equation} 
\end{teo}


\nl
Before turning to the proofs of the above statements we make the following remarks:

\begin{enumerate}

\item In the same spirit of \cite{shatah1,shatah2} the construction of the energy \eqref{Energy} is based 
on an evolution equation for $\mD_{t_+} \k_+$; see \eqref{D_t^2kappa}.

\item 
Proposition \ref{teoVS1} is the analogous of 
proposition 4.3 in \cite{shatah1} (one fluid problem with vanishing surface tension)
and proposition 4.3 in \cite{shatah2} (interface problem).
Since our energy is based exclusively on $v_+$, 
and we cannot take full advantage of the presence of surface tension - its highest Sobolev norm being not uniformly controlled -
we can only establish the weighted weaker control \eqref{venergy-} on $v_-$.
Under condition \eqref{condition} this 
turns out to be still sufficient to obtain uniform energy estimates.

\item
Theorem \ref{teoEE} is the analogous of theorem 4.4 in \cite{shatah1} and theorem 4.5 in \cite{shatah2}.
The proof uses essentially the same techniques.

\item 
{\it Convergence of solutions}. An immediate corollary of the uniform energy estimates 
provided by theorem \ref{teoEE} is weak-star convergence
of solutions of \eqref{E}-\eqref{BC}
with outer density and surface tension tending to zero,
to solutions of the water wave problem for one fluid in vacuum
without surface tension \eqref{E_0}-\eqref{BC_0}.
Weak convergence in a larger Sobolev space can also be obtained easily in Lagrangian coordinate,
writing the integral equation for \eqref{E}-\eqref{BC} 
and passing to the limit using standard Rellich compactness.

\item
{\it The case $\e = 1$}. In the case of constant surface tension's strength
we recover the result obtained in \cite{CoutShko3} and independently by the author in \cite{VS}.

\item
\label{corconv}
Using the non-linear Eulerian framework introduced in \cite{shatah1,shatah2} 
it is not hard to obtain compactness in time for solutions of \eqref{E}-\eqref{BC}
and therefore strong convergence
to solutions of \eqref{E_0}-\eqref{BC_0}.
A more precise statement is the following:

\begin{cor}[\bf{Convergence of solutions}].
\label{corconvergence}
Let an initial hypersurface $S_0 \in H^{3k}$ and an initial velocity field $v_0 \in H^{3k} (\Omega_0)$ 
be given for some integer $k$ with $3k > \frac{n}{2} + 2$.
Consider any sequence of local-in-time solutions
\begin{equation}
S_t^m \in C ( [0,T] ; H^{3k} )
\hskip 8pt , \hskip10pt
v^m \in C ( [0,T] ; H^{3k} (\Omega_t^m) ) 
\end{equation}
of (\ref{E})-(\ref{BC}) 
corresponding to densities $\rho^m = \rho_+  \chi_{\Omega_t^+}   +   \rho_-^m \chi_{\Omega_t^-}$
and surface tension's strength $\e^2_m$. 
Let $u^m$ be the Lagrangian map corresponding to the velocity field $v^m$
and suppose that $\rho^m_- , \e_m  \rightarrow 0$ as $m \rightarrow \infty$
under the constraint $\rho_-^m \leq \e_m^{3/2}$.  
Then there exist a small positive time $T^\infty$,
a map $u^\infty$,
and a vector field $v^\infty$
such that the following is true for any $k^\p < k$:
\begin{eqnarray*}
& 1) & \lim_{m \rightarrow \infty} u_+^m  =  u^\infty  \hskip 10pt
						\mbox{in} \hskip12pt C \left( [0, T^\infty] ; H^{3k} (\Omega_0^+) \right)
\nonumber
\\
&   &  \lim_{m \rightarrow \infty} v_+^m \circ u_+^m  =  v^\infty \circ u^\infty \hskip 10pt
						\mbox{in} \hskip12pt C \left( [0, T^\infty] ; H^{3k^\p} (\Omega_0^+) \right)
\nonumber
\\
& 2) & 	S_t^\infty :=  \partial \Omega_t^\infty := \partial u^\infty (t , \Omega_0) \in H^{3k^\p} 
\nonumber
\\
& 3) &  \mbox{ $(v^\infty , S_t^\infty)$  are a strong (pointwise) solution of \eqref{E_0}-\eqref{BC_0} 
						for $t \in [0,T^\infty] \, $.}
\end{eqnarray*}
\end{cor}

\end{enumerate}

%
%

\section{Proofs of the statements}
\label{secproof}

\subsection{Preliminary Estimates}
\label{secpreliminaries}


Let us denote by $Q$ any generic polynomial 
with positive coefficients (depending on the set $\Lambda_0$), independent of $\rho_-$ and $\e$,
whose arguments are quantities that will be bounded by the energy through proposition \ref{teoVS1}, i.e.,
\begin{equation}
\label{Q}
Q = Q \left( {|v_+|}_{H^{3k} (\Omega_t^+)}, \sqrt{\e} {|v_-|}_{H^{3k} (\Omega_t^-)}, {|v_-|}_{H^{3k-1} (\Omega_t^-)} ,
\e {|\k_+|}_{H^{3k - 1} (S_t)}, {|\k|}_{H^{3k-2} (S_t)} \right)   \,  .
\end{equation}
From \eqref{stimaNPI}, trace estimates, and interpolation of Sobolev norms, the
following quantities can also be bounded by  $Q$: 
\begin{align*} 
& {|\Pi_\pm|}_{H^{3k-2} (S_t) } \hskip4pt , \hskip6pt 
{|N_\pm|}_{H^{3k-1} (S_t) }  \hskip4pt , \hskip6pt 
\sqrt{\e} {|\k_+|}_{H^{3k-\frac{3}{2}} (S_t)} \hskip4pt , \hskip6pt 
\\
& {|v_+^\top|}_{ H^{3k-\frac{1}{2}} (S_t) }  \hskip4pt , \hskip6pt  
\sqrt{\e} {|v_-^\top|}_{H^{3k-\frac{1}{2}} (S_t) } \, .
\end{align*}

\vskip5pt
\begin{lem}[\bf{Estimates for the pressure}]
\label{lemp}
Let $p_\k$ be defined by \eqref{Sp}.
There exists a positive constant $C$, depending only on the set of hypersurfaces $\Lambda_0$, such that
\begin{align}
\label{estimatep_k}
{| \nabla p_\k |}_{H^{3k - \frac{5}{2}} (\R^n \minus S_t)}  & \leq  
						Q  
\\
\label{estimateep_k}
\e {| \nabla p_\k |}_{H^{3k - \frac{3}{2}} (\R^n \minus S_t)}  & \leq  
Q
\, .
\end{align}
Let $p^\star_{v,v}$ and $p_{v,v}$  be defined respectively by \eqref{pwaterwave} and \eqref{p_vw}.
Then
\begin{eqnarray}
\label{estimatep^star_vv}
& & {| \nabla_{\H_+ N_+} p^\star_{v,v} |}_{ H^{3k - \frac{1}{2}} (\Omega_t^+) }  +
			{| D^2 p^\star_{v,v} |}_{ H^{3k - \frac{3}{2}} (\Omega_t^+) }  \leq Q \, ,
\\
\label{estimatep^S}
& & {| \nabla \H_\pm p^S_{v,v} |}_{H^{3k-\frac{3}{2}} (\Omega_t^\pm)}  \leq Q \, ,
\\
\label{estimatep_vv}
& &  {| \nabla p_{v,v}^+ |}_{H^{3k - \frac{3}{2}} (\Omega_t^+)} \hskip7pt ,\hskip10pt
{| \nabla p_{v,v}^- |}_{H^{3k - \frac{3}{2}} (\Omega_t^-)}  \leq  Q  
\, .
\end{eqnarray}
and, as a consequence,
\begin{eqnarray}
\label{estimateD_tv}
{| \mD_t v |}_{H^{3k - \frac{3}{2}} (\R^n \minus S_t)}    \leq  Q   \,  .
\end{eqnarray}
\end{lem}

\begin{proof}
The first two estimates follow by the definition of $p_\k$ and lemma \ref{lemmaoperator1} and \ref{lemmaoperator2}.
\eqref{estimatep^star_vv} is proved in \cite[lemma 4.8]{shatah1}.
Using the explicit expression for $p^S_{v,v}$ in \eqref{p_vw}, 
\eqref{stimaNPI}, and again lemma \ref{lemmaoperator1} and \ref{lemmaoperator2} 
together with product Sobolev-estimates\footnote{
An estimate we use several times throughout our proofs is
\begin{equation*}
{|f g|}_{ H^{s_1} (S) } \leq C {|f|}_{ H^{s_1} (S) } {|g|}_{ H^{s_2} (S) }
\end{equation*}
for $s_2 \geq s_1$, $s_2 > (n-1)/2$, and $s_1 + s_2 \geq 0$.
}, 
we see that 
for any $0 \leq s \leq 3k - 1$
\begin{align}
\nonumber
{| \nabla \H_\pm p^S_{v,v} |}_{H^{s} (\Omega^\pm_t)}  \leq 
		C \rho_- {| a |}_{H^{s - \frac{1}{2}} (S_t)}
		& \leq
		C \rho_-  \left( 1 + {|\k_+|}_{H^{s - \frac{1}{2}} (S_t)}  \right)  
		{| v |}^2_{H^{3k-1} (\R^n \minus S_t)}  
		\\
		& +
		\label{estimatep^S1}
		C \rho_- {| v_+^\bot |}^2_{H^{3k - \frac{1}{2}} (S_t)} {| v |}_{H^{3k-1} (\R^n \minus S_t)}
		\, .
\end{align}
This proves \eqref{estimatep^S}.
Using the identity $f_\pm = \Delta^{-1}_\pm \Delta f_\pm  +   \left. \H_\pm  f_\pm \right|_{S_t}$,
we can write
\begin{equation*}
\nabla p_{v,v}^\pm   =   - \nabla \Delta^{-1}_\pm \tr{(Dv)}^2  +  \frac{1}{\rho_\pm} \nabla \H_\pm p^S_{v,v}  \,  ,
\end{equation*}
so that \eqref{estimatep^S1} implies \eqref{estimatep_vv}.
To conclude we notice that \eqref{estimateD_tv} 
follows directly from \eqref{Ematerial}, \eqref{estimateep_k} and  \eqref{estimatep_vv}
\end{proof}

\vskip5pt
\begin{lem}
\label{derivationRT}
Let $p^\star_{v,v}$ be defined by \eqref{pwaterwave}, then
\begin{equation*}
{|  N_+ \cdot \Delta_{S_t} \nabla p^\star_{v,v}	-  \nabla_{N_+} p^\star_{v,v} \N_+ \k_+  |}_{H^{3k - \frac{5}{2}} (S_t)} \leq Q
\end{equation*}
\end{lem}

\nl
The proof of this lemma is based on the decomposition of the Laplacian on $S_t$:
$\Delta f = \Delta_{S_t} f   +  \k_+ \nabla_{N_+} f   +   D^2 f (N_+,N_+) $.
Details can be found in \cite[721-722]{shatah1}.
%
%
%
%
%
%
%
%
The following lemma is the key to our energy estimates and is 
the analogous for the two-phase problem of lemma 3.4 in \cite{shatah1}.
\begin{lem}
\label{lemD_t^2}
Let $S_t \in H^{3k}$, with $S_t \in \Lambda_0$, 
and $v \in H^{3k} (\R^n \minus S_t)$ be a solution to \eqref{E}-\eqref{BC},
then
\begin{equation}
\label{D_t^2kappa}
{\left| \mD_{t_+}^2 \k_+  - \e^2 \Delta_S \bar{\N} \kappa_+  - \frac{1}{\rho_+} \Delta_S \N_+ p^S_{v,v}
			- (\rho_+  +  \rho_-) \nabla_{N_+} p^\star_{v,v} \bar{\N} \k_+
			\right|}_{ H^{3k - \frac{5}{2}} (S_t) } \leq Q \, ,
\end{equation}
provided $\rho_- \leq \e$.
\end{lem}

\begin{proof}
Using \eqref{D_tkappa} together with \eqref{D_tN}, \eqref{D_tPI} and commutator estimate \eqref{commest1} we get
\begin{equation*}
{\left| \mD_{t_+}^2 \k_+  +  N_+ \cdot \Delta_S \mD_{t_+} v_+ - 
		2 \Pi \cdot (( \left. D^\top \right|_{T\partial\Omega_t}) \mD_{t_+} v_+ ) 
		\right|}_{ H^{3k - \frac{5}{2}} (\Omega_t) } \leq Q \ .
\end{equation*}
Using Euler's equation \eqref{Ematerial},
$ p^+_{v,v} = p^\star_{v,v} + \H_+ \left. p^+_{v,v} \right|_{S_t}$, and $N_+ \cdot \nabla p_\k^+ = \bar{\N} \k_+$, 
we can write
\begin{eqnarray*}
& &  N_+ \cdot \Delta_{S_t} \mD_t v_+  -  2 \Pi \cdot (( \left. D^\top \right|_{T\partial\Omega_t}) \mD_t v_+ ) 
						\\
						& =	&  - N_+ \cdot \Delta_{S_t} \nabla p^+_{v,v} 
						+  2 \Pi \cdot (( \left. D^\top \right|_{T\partial\Omega_t}) \nabla p^+_{v,v} )
						\\
						& & - \e^2 N_+ \cdot \Delta_{S_t}  \nabla p_\k^+ 
						+  2 \e^2 \Pi \cdot (( \left. D^\top \right|_{T\partial\Omega_t} \nabla p_\k^+) 
						\\
						& = &  
						- N_+ \cdot \Delta_{S_t} ( \nabla p^\star_{v,v} )  
						+  2 \Pi \cdot (( \left. D^\top \right|_{T\partial\Omega_t}) \nabla p^\star_{v,v} )
						\\
						& & -  \Delta_{S_t} \N_+ p^+_{v,v}
						+  \Delta_{S_t} N_+ \cdot \nabla \H_+ \left.  p^+_{v,v} \right|_{S_t}
						-  \e^2 \Delta_{S_t}  \bar{\N} \k_+   +  \e^2 \Delta_{S_t} N_+ \cdot \nabla p_\k^+  \, .
\end{eqnarray*}
Using \eqref{stimaNPI}, \eqref{estimateep_k}, and the identity $\Delta_{S_t} N_+ = {|\Pi_+|}^2 N_+  +  \nabla^\top \k_+$,
we can estimate
\begin{eqnarray*} 
{ \left|  \Pi \cdot (( \left. D^\top \right|_{T\partial\Omega_t}) \nabla p^\star_{v,v} )  
			\right| }_{H^{3k - \frac{5}{2}} (S_t) }  \hskip 4pt ,  \hskip 8pt 
			\e^2 { | \Delta_{S_t} N_+ \cdot \nabla p_\k^+   | }_{H^{3k - \frac{5}{2}} (S_t) }   \leq Q  \, .
\end{eqnarray*}
From \eqref{estimatep^S1}, we also see that assuming $\rho_- \leq \e$ gives
\begin{eqnarray*} 
{ \left|  \Delta_{S_t} N_+ \cdot \nabla \H_+ \left.  p^+_{v,v} \right|_{S_t}  \right| }_{H^{3k - \frac{5}{2}} (S_t) }   
																								\leq  Q  \, .
\end{eqnarray*}
Combining these estimates with the above chain of identities, lemma \ref{derivationRT}, and \eqref{N_+barN},
gives \eqref{D_t^2kappa}
\end{proof}

%
%

\subsection{Proof of proposition \ref{teoVS1}}
\label{secproofteoVS1}

\nl
$\bullet$ Proof of \eqref{ekappaenergy} -
The estimates on the mean-curvature $\kappa_+$ follow easily from the definition of $E_2$ and $E_{RT}$,
respectively in \eqref{E_2} and \eqref{E_RT},
and the properties of $\bar{\N}$ in lemma \ref{lemmaoperator2}.


\nl
$\bullet$  Proof of \eqref{venergy+} -
To estimate $v_+$  we use the fact\footnote{
An essential proof of this fact is contained in \cite[pp. 717-719]{shatah1} and \cite[pp. 864-865]{shatah2}.
See also \cite[Appendix]{shatah3} for further discussion.
}
that for $\partial \Omega \in \Lambda_0$ and $1/2  < s \leq 3k$
\begin{equation}
\label{vectorcontrol}
{ | w | }_{H^s (\Omega) } \leq C \left( { |\div w| }_{H^{s-1} (\Omega) } + { |\curl w| }_{H^{s-1} (\Omega) }
	+  {| \Delta_{\partial \Omega} w \cdot N_+| }_{
	H^{s-\frac{5}{2}} (\partial \Omega) }
	+  { |w| }_{L^2 (\Omega) }  \right)
\end{equation}
where the constant $C$ only depends on $\Lambda_0$.
Since $v_+$ is divergence-free, and the vorticity $\o_+$ is included in the energies, 
we only need to control the boundary value of $v_+$.
From the definition of $E_1$ in \eqref{E_1}, and the properties of $\bar{\N}$, it is clear that
\begin{equation*}
{ \left| \mD_{t_+} \kappa_+  \right| }^2_{ H^{3k - \frac{5}{2}} (S_t) } \leq C ( 1 + E_1 ) \, .
\end{equation*}
From \eqref{D_tkappa} we have
\begin{align*}
{ | - \Delta_{S_t} v_+ \cdot N_+ |  }^2_{H^{ 3k - \frac{5}{2} } (S_t)}  
				& \leq
				C \left( 1 + E_1 +  {|v_+|}_{H^{ 3k - \frac{1}{8} } (\Omega_t^+)} \right)
				\\
				& \leq
				C ( 1 + E_1 )
				+ \beta {|v|}_{H^{3k} (\Omega_t)}
				+ C \beta^{-1} {|v|}_{L^2 (\Omega_t)}
\end{align*}
for some parameter $\beta > 0$. 
Choosing $\beta$ small enough, and controlling ${|v_+|}_{L^2}$ by $E_0$,
gives \eqref{venergy+}.


\nl 
$\bullet$ Proof of \eqref{venergy-} - This estimate is proved in four steps.
\newline
1) {\it Estimates on the Lagrangian coordinate map}.
Let $u_\pm$ denote the solution of \eqref{lagrangianmap}.
Using product Sobolev estimates it is not hard to see that 
\begin{align*}
{| u_+ (t,\cdot)  -  \id_{\Omega_0^+} |}_{H^{3k} (\Omega_0^+) } & \leq
        C_1  \int_0^t {| v_+ (s, \cdot) |}_{H^{3k} (\Omega_t^+) }
        {|u_+ (s, \cdot) |}_{H^{3k} (\Omega_0^+) }^{3k} \, ds \, ,
\\
{| u_- (t,\cdot)  -  \id_{\Omega_0^-} |}_{H^{3k-1} (\Omega_0^-) } & \leq
        C_1  \int_0^t {| v_- (s, \cdot) |}_{H^{3k-1} (\Omega_t^-) }
        {|u_- (s, \cdot) |}_{H^{3k-1} (\Omega_0^-) }^{3k-1} \, ds  \, ,
\end{align*}
where $C_1 > 0$ only depends on $n$ and $k$.
Next, we let $\mu$ be a sufficiently large constant compared to the initial data, and define
\begin{equation}
\label{t_0}
t_0 := \sup  \left\{ t \, : \, {|v_+ (s, \cdot)|}_{H^{3k} (\Omega_s^+) } + 
                               {|v_- (s, \cdot)|}_{H^{3k-1} (\Omega_s^-) }  \leq \mu \, \for s \in [0,t]  \right\} \, .
\end{equation}
Since $v$ is assumed to be continuous in time with values in $H^{3k}$, $t_0 > 0$.
An ODE argument based on Gronwall's inequality 
shows that there exists a positive time $t_1$ and a constant $C_2$, only depending on $k,n$,$\mu$ and $\Lambda_0$,
such that
\begin{equation}
\label{t_1}
{| u_+ (t, \cdot) - \id_{\Omega_0^+} | }_{H^{3k} (\Omega_0^+) } 
+
{| u_- (t, \cdot) - \id_{\Omega_0^-} | }_{H^{3k-1} (\Omega_0^-) }  \leq C_2 t  \leq \frac{1}{2}
\end{equation}
for any $t \in [0, t^\star]$, where $t^\star := \min\{t_0,t_1, 1/(2C_2) \}$ depends only on $\Lambda_0$ and the initial data.
This in particular shows that $u_\pm$ is a diffeomorphism, so that $u_\pm^{-1} (t, \cdot)$ is 
a well-defined volume preserving map for $x \in \Omega_t^\pm$, and
for the same range of times we have
\begin{equation}
\label{Du-}
{| {(D u_+ )}^{-1} | }_{H^{3k-1} (\Omega_0^+) } \quad , \quad {| {(D u_- )}^{-1} | }_{H^{3k-2} (\Omega_0^-) } 
\leq   2  \, .
\end{equation}

\nl
2) {\it Decomposition of vector fields and control of ${|v_-|}_{L^2}$}.
The well-know Hodge decomposition of vector fields allows one to decompose any arbitrary vector field $w$,
defined on a domain $\Omega \subset \R^n$, in two components, a divergence-free component and a gradient part.
More precisely we can write $w = v + \nabla g$, where $\div v = 0 = v^\bot$, and
$g$ satisfies the Neumann boundary problem
\begin{equation*}
\left\{
\begin{array}{ll}
\Delta g = \div w   &,  \, x \in \Omega
\\
\nabla_{N} g = w^\bot  &,  \,  x \in \partial \Omega \, .
\end{array}
\right.
\end{equation*}
We denote by $w_{ir} := \nabla g$ the so-called irrotational part of $w$ and 
define the projection $P_r$ on the rotational part by $w_r := P_r (w) := w - w_{ir}$.
This splitting is orthogonal on $L^2$ and $P_r (w)$ is a gradient-free projection\footnote{
More details on this decomposition and related estimates are given in \cite[Appendix]{shatah3}.
}.
If we consider the divergence-free velocity field $v_-$, the above decomposition reduces to
\begin{equation*}
v_-   =  \nabla \H_- \N_-^{-1} v_-^\bot  +  v_{-, r} \, .
\end{equation*} 
In \cite{shatah3} it is observed that the invariance of Euler equations under the action of
the group of volume preserving diffeomorphisms leads, via Noether's theorem, to a family of conserved quantities
which determine completely the rotational part of the velocities\footnote{
For completeness we provide here the proof.
Consider $ F = {( D u )}^\star (v \circ u) $, the pullback of $v$ by the map $u$. 
Taking a time derivative, using Euler equations $\partial_t (v \circ u) = - \nabla p \circ u$
and \eqref{lagrangianmap} we get
\begin{equation*}
\frac{d}{dt} F = \frac{1}{2} \nabla {| v \circ u|}^2  -  {( D u )}^\ast \nabla p \circ u
\end{equation*}
hence
\begin{equation*}
F (t) = F(0) + \nabla \left( \int_0^t \frac{1}{2} {| v \circ u|}^2  -  p \circ u \right)
\end{equation*}
which in turn implies
\begin{equation}
\label{P_r}
v (t, x)  =   {\left( D u^{-1} \right)}^\ast v (0 , u^{-1} (t,x) )  +  \nabla f
\end{equation}
for some $f$, and therefore proves \eqref{consequenceP_r}.
}:
\begin{equation}
\label{consequenceP_r}
v_r (t,\cdot) = P_r \left( S_t, {(Du^{-1})}^\ast v(0, u^{-1}( t, \cdot) )  \right)
\end{equation}
where $P_r (S_t, w)$ denotes the projection of $w : \R^n \minus S_t \rightarrow \R^n$ onto its
rotational (gradient-free) part.
Applying the above identity to $v_-$, using standard estimates for the elliptic Neumann-problem, 
$v_-^\bot = - v_+^\bot$, and \eqref{Du-},  we can estimate
\begin{eqnarray*}
{ | v_-  | }^2_{L^2 (\Omega_t^-)} & = & { | v_r | }^2_{L^2 (\Omega_t^-)} + { | v_{ir} | }^2_{L^2 (\Omega_t^-)} 
\\
& \leq & {| {(Du_-^{-1})}^\ast v(0, u^{-1}( t, \cdot))  |}^2_{L^2 (\Omega_t^-)}
                   + { | v_+  | }^2_{L^2 (\Omega_t^+)} 
\\
& \leq &  C {| Du_-^{-1} |}_{L^\infty (\Omega_0^-)}^2 {|  v(0, \cdot)  |}_{L^2 (\Omega_0^-)}^2 + C E_0  
                   \leq C (1 + E_0)
\end{eqnarray*}
with $C$ depending only on the initial data.

\nl
3) {\it  Control of ${|v_-|}_{H^{3k-1}}$}.
For this purpose we want to apply the following variant of \eqref{vectorcontrol}:
\begin{align}
\nonumber
{ | w | }_{H^s (\Omega) } \leq C (1 + {|\k_+|}_{H^{s - \frac{3}{2}}} ) 
	& \left( { |\div w| }_{H^{s-1} (\Omega) } + { |\curl w| }_{H^{s-1} (\Omega) } \right.
	\\
	& \left. +  {| w^\bot | }_{ H^{s-\frac{1}{2}} (\partial \Omega) }
	+  { |w| }_{L^2 (\Omega) } 
	\right)
\label{vectorcontrol2}
\end{align}
for $1/2 < s \leq 3k$.
To control the vorticity term $\curl v_-$, we use \eqref{P_r} 
and the fact that pull-backs commute with exterior derivatives to get
\begin{equation*}
\curl v_- (t, \cdot) =   {(Du^{-1})}^\ast \curl v_- (0, u^{-1}( t, \cdot) ) \, .
\end{equation*}
Then, \eqref{Du-} implies
\begin{equation}
\label{curlv-}
{| \curl v_- |}_{ H^{s} (\Omega_t^-) }  \leq  C \hskip10pt , \hskip7pt 0 \leq s \leq 3k-2
\end{equation}
for some constant $C$ depending only on the initial data.
Using the above inequality with $s = 3k-2$, and \eqref{vectorcontrol2} together with $v_-^\bot = - v_+^\bot$,
we have
\begin{align*}
{|v_-|}^2_{ H^{3k-1} (\Omega_t^-) }  & \leq C \left(  {| v_+^\bot  |}^2_{ H^{3k- \frac{3}{2}} (S_t) }
					+
					{|\curl v_-|}^2_{ H^{3k-2} (\Omega_t^-) }  
					+
 					{|v_-|}^2_{ L^2 (\Omega_t^-) }
 					\right)
\\
& \leq  C ( 1 + E + E_0) 
\end{align*}
with $C$ depending only on $\Lambda_0$ and the initial data.

\nl
4) {\it  Weighted control of ${|v_-|}_{H^{3k}}$}.
We want to use \eqref{vectorcontrol} with $s = 3k$.
Notice that the vorticity term $\e {| \o_- |}^2_{H^{3k-1}(\Omega_t^-)}$ is already included in the energy \eqref{Energy},
and that ${|v_-|}_{L^2(\Omega_t^-)}$ has been estimated in the previous paragraph.
Therefore, in order to conclude the proof of \eqref{venergy-},
we just need to control the boundary value of $v_-$.
Since $v_-^\bot = - v_+^\bot$, we have
\begin{align*}
N_- \cdot \Delta_{S_t} v_-  & =  - \Delta_{S_t} v_+^\bot   -   2 \Pi \cdot ( \left. D^\top \right|_{S_t} v_- )
	- v_- \cdot \Delta_{S_t} N_-
	\\
	& =  - N_+ \cdot \Delta_{S_t} v_+  -   2 \Pi \cdot ( \left. D^\top \right|_{S_t} (v_+ + v_-) )
	- (v_+ + v_-) \cdot \Delta_{S_t} N_- \, ,
\end{align*}
so that
\begin{align*}
{| N_- \cdot \Delta_{S_t} v_- |}^2_{ H^{3k- \frac{5}{2}} (S_t) }  
& \leq C ( 1 + E_1 ) 
\\
& + C {| v |}^2_{ H^{3k - 1} (\R^n \minus S_t) } 
\left( {|\Pi|}^2_{ H^{3k- \frac{3}{2}} (S_t) }  +  {|N|}^2_{ H^{3k- \frac{1}{2}} (S_t) } \right) 
\\
& 
\leq
C ( 1 + E_1 )
	+ C {| v |}^2_{ H^{3k - 1} (\R^n \minus S_t) } 
	\left( 1 + {|\k_+|}^2_{ H^{3k- \frac{3}{2}} (S_t) } \right) \, ,
\end{align*}
having used \eqref{stimaNPI}.
Finally, interpolating $\k_+$ between $H^{3k - \frac{5}{2}}$ and $H^{3k-1}$, and using
\eqref{ekappaenergy}, we have
\begin{equation*}
\e {|\k_+|}^2_{ H^{3k - \frac{3}{2}} (S_t) } \leq  C (1 + E) \, ,
\end{equation*}
which combined with the previous estimate gives
\begin{eqnarray*}
\e {| N_- \cdot \Delta_{S_t} v_- |}^2_{ H^{3k- \frac{5}{2}} (S_t) }
			\leq   	{  C ( 1 + E + E_0)  }^{2} \, .
\end{eqnarray*}
This concludes the proof of \eqref{venergy-}  $ _\blacksquare $

%
%
%
%

\subsection{Proof of Theorem \ref{teoEE}}
\label{secproofteoEE}

\subsubsection{Estimate on  ${|\k|}_{H^{ 3k-\frac{5}{2}} (S_t)}$}
The estimate on the Lagrangian coordinate map in \eqref{t_1} 
implies in particular the estimate on the mean-curvature
\begin{equation}
\label{estklow}
{| \k_+ (t, \cdot) | }_{H^{3k - \frac{5}{2} } (S_t) } \leq C t + 
		{| \k_+ (0, \cdot) | }_{H^{3k - \frac{5}{2} } (S_0) }
 		\hskip10pt \for t \in [0, \min\{t_0, t_1 \}] \, ,
\end{equation}
where the constant $C$ is only determined by $\mu$ (see \eqref{t_0}) and the set\footnote{
This can be checked using the local coordinates constructed in \cite[appendix A]{shatah1}.
} 
$\Lambda_0$.
We conclude that there exists a time $t_2$, determined again only by $\mu$ and the set $\Lambda_0$,
such that 
\begin{equation*}
S_t \in \Lambda_0 \hskip8pt , \hskip8pt  \for t \in [ 0, \min\{t_0,t_2\}] \, .
\end{equation*}

\subsubsection{Evolution of the Energy}

The following proposition shows how the time evolution of $E$ can be bounded by a polynomial $Q(E)$ 
up to the time derivative of an extra energy term
due to the Kelvin-Helmotz instability.
\begin{pro}
Assuming $\rho_- \leq \e^{3/2}$, there exists a polynomial $Q$, as in \eqref{Q},
with positive coefficients depending on the set $\Lambda_0$ and independent of $\rho_-$ and $\e$,
such that 
\begin{equation}
\label{evolutionE}
\left|  \frac{d}{dt} (E - E_{\ex}) \right| \leq Q \, ,
\end{equation}
where the extra energy term $E_{\ex}$ is given by 
\begin{eqnarray}
\label{E_ex}
E_{\ex}	=	 - \frac{\rho_-}{ 2 (\rho_+ + \rho_-) } 
			\int_{S_t} \nabla_{v_+^\top - v_-^\top} \kappa_+ \cdot \bar{\mathcal{N}} 
			{(- \Delta_{S_t} \bar{\mathcal{N}}) }^{2k-2} \nabla_{v_+^\top - v_-^\top} \kappa_+   \, dS \, .
\end{eqnarray}
\end{pro}

\pr
Combining \eqref{D_tdS}
with the divergence decomposition formula 
\begin{equation*}
\left.  \div v_\pm \right|_{S_t} =  \mathcal{D} \cdot v_\pm^\top  + \k_\pm v_\pm^\top + \nabla_{N_\pm} v_\pm \cdot N_\pm  =  0 \, ,
\end{equation*}
we see that
\begin{equation*}
\mD_{t_\pm} dS = - \nabla_{N_\pm} v_\pm \cdot N_\pm dS \, .
\end{equation*}
Then, since $3k - \frac{5}{2} > \frac{n-1}{2}$, we can bound
\begin{equation*}
{ \left| \nabla_{N_\pm} v_\pm \cdot N_\pm  \right| }_{L^\infty (S_t)} \leq C  {|v_\pm|}_{H^{3k-1} (\Omega_t^\pm)}  \leq Q \, .
\end{equation*}
Therefore, $\mD_{t_\pm} dS$ will not complicate the estimates.
We now proceed to analyze the time evolution of each one of the terms in the energy \eqref{Energy}
keeping track only of terms which cannot be bounded by $Q$.


\nl
$\bullet$ {\it Evolution of $E_{RT}$}:
We want to show 
\begin{equation}
\label{evolutionE_RT}
\left| \frac{d}{dt} E_{RT}  +
			  (\rho_+  +  \rho_- )
			  \int_{S_t} \nabla_{N_+} p^\star_{v,v}  \bar{\N} \k_+  \, \bar{\N} {( - \Delta_{S_t} \bar{\mathcal{N}}  )}^{2k-2}
			  \mathbf{D}_{t_+} \k_+
			  \, dS
			  \right|  \leq Q \,  .
\end{equation}
From the commutator estimate \eqref{D_tgrad}, formula \eqref{D_tN},
and the definition of $p^\star_{v,v}$ in \eqref{pwaterwave},
we get
%
%
\begin{align*}
\mD_{t_+} \nabla_{N_+} p^\star_{v,v}
& =  - N_+ \cdot \left( {(Dv_+)}^\ast \nabla p^\star_{v,v}  -  \nabla \mD_{t_+} p^\star_{v,v} \right) 
\\
& =
- N_+ \cdot \left( {(Dv_+)}^\ast \nabla p^\star_{v,v}  +  \nabla \Delta^{-1}_+ \mD_{t_+} \tr{(Dv)}^2 
      + [\mD_{t_+}, \Delta^{-1}_+] \tr{(Dv)}^2  \right)
\end{align*}
Using Euler's equations we see that
$\mathbf{D}_{t_+} \tr {(Dv_+)}^2 = - 2 \tr [ {(Dv_+)}^3  - 2 \rho_+ D^2 p_+ \cdot D v_+]$.
Combining this with \eqref{D_tDelta-1}
gives the estimate
\begin{equation*}
{| \mD_{t_+} \nabla_{N_+} p^\star_{v,v} | }_{L^\infty (S_t)} \leq Q  \, .
\end{equation*}
Then, we see from the definition of $E_{RT}$ in \eqref{E_RT}, and 
commutator estimates \eqref{commest1} and \eqref{commest2}, that
\begin{align*}
\left| \frac{d}{dt} E_{RT}  \right.  +   (\rho_+  + \rho_-) 
			\left. \int_{S_t} 
			{(-\bar{\N} \Delta_{S_t} )}^{k-1}  \bar{\N} \mD_{t_+} \k_+
			\nabla_{N_+} p^\star_{v,v}   {(-\bar{\N} \Delta_{S_t} )}^{k-1}  \bar{\N} \k_+  \, dS
			\right| \leq Q \, .
\end{align*}
This already gives \eqref{evolutionE_RT} in the case $k=1$.
For $k \geq 2$, we use lemma \ref{lemmacommp^star} to commute the multiplication operator
by $\nabla_{N_+} p^\star_{v,v}$ with $\bar{\N}$ and $ \Delta_{S_t}$,
and finally obtain \eqref{evolutionE_RT}.


\nl
$\bullet$ {\it Evolution of $E_2$}:  
From the definition of $E_2$ in \eqref{E_2}, and commutator estimates
\eqref{commest1} and \eqref{commest2}, it follows
\begin{equation}
\label{evolutionE_2}
\left| \frac{d}{dt} E_2   + \e^2  \int_{S_t}  
         \k_+ \bar{\mathcal{N}} {(- \Delta_{S_t} \bar{\mathcal{N}}) }^{2k-1} \mathbf{D}_{t_+} \k_+ \, dS \right| \leq Q  \,  .
\end{equation}


\nl
$\bullet$ {\it Evolution of the vorticity $\o= Dv - {(Dv)}^\star$}: 
Commuting $\mD_{t_\pm}$ and $D$ we get the identity
\begin{align}
\nonumber
\mathbf{D}_{t_\pm} \o_{\pm} & = D \mathbf{D}_{t_\pm} v_\pm - {(Dv_\pm)}^2 - {(D \mathbf{D}_{t_\pm} v_\pm)}^\ast
              +  {({(Dv_\pm)}^\ast)}^2
\\
& = {({(Dv_\pm)}^\ast)}^2 - {(Dv_\pm)}^2 = - \o_\pm Dv_\pm  - {(Dv_\pm)}^\ast \o_\pm \, .
\end{align}
Then, repeated commutations and product Sobolev estimates
show that, for any integer $0 \leq s \leq 3k$,
\begin{align}
\nonumber
\int_{\Omega_t^\pm}  \mathbf{D}_{t_\pm} {|D^s \o_\pm |}^2 \, dx
              & \leq  C {|\o_\pm (t, \cdot)|}^2_{ H^s (\Omega_t^\pm) } {|Dv_\pm (t, \cdot)|}_{ L^\infty (\Omega_t^\pm)  }
              \\
              & +
              C {|\o_\pm (t, \cdot)|}_{ H^s (\Omega_t^\pm) }  {|Dv_\pm (t, \cdot)|}_{ H^s (\Omega_t^\pm)  } 
              {|\o_\pm (t, \cdot)|}_{ L^\infty (\Omega_t^\pm) } \, .
\label{D_to}
\end{align}
In the case of $\o_+$, we use the above inequality with $s = 3k-1$ 
to get
\begin{align}
\label{evolvorticity}
\frac{d}{dt}  \int_{\Omega_t^+} {|D^{3k -1}  \o_+ |}^2  \, dx  \leq
                {|v_+ (t, \cdot)|}_{ H^{3k} (\Omega_t^+)  } 
                {|\o_+ (t, \cdot)|}_{ H^{3k-1} (\Omega_t^+) }^2  \leq Q \, .
\end{align}
In the case of $\o_-$, we use again \eqref{D_to},
\eqref{venergy-low}, \eqref{venergy-},
and \eqref{curlv-} together with Sobolev's embedding, to obtain
\begin{align}
\nonumber
\e & \frac{d}{dt}  \int_{\Omega_t^-} {|D^{3k -1}  \o_- |}^2  \, dx
\\
\nonumber
& \leq C \e {|\o_- (t, \cdot)|}_{ H^{3k-1} (\Omega_t^-) }^2 {|v_- (t, \cdot)|}_{ H^{3k-1} (\Omega_t^-)  }
\\
& + C \sqrt{\e} {|\o_- (t, \cdot)|}_{ H^{3k-1} (\Omega_t^-) }  \sqrt{\e} {|v_- (t, \cdot)|}_{ H^{3k} (\Omega_t^-)  }
              {|\o_- (t, \cdot)|}_{ L^\infty (\Omega_t^-) }
\label{evolvorticity-}
\leq Q \, .
\end{align}

%
%

\nl
$\bullet$  {\it Evolution of $E_1$}:
From the definition of $E_1$ in \eqref{E_1}, commutator estimates \eqref{commest1} and \eqref{commest2}
we have
\begin{equation*}
\left| \frac{d}{dt} E_1	-	\int _{S_t}  
			\bar{\N} \mD_{t_+}^2 \kappa_+ {(-  \Delta_{S_t} \bar{\N}  ) }^{2k-2} 
			\mD_{t_+} \k_+ \, dS
			\right| \leq Q  \, .
\end{equation*}			
Using \eqref{D_t^2kappa} we get
\begin{align*}
\left| \frac{d}{dt} E_1	\right. & -	\int _{S_t} 
			\bar{\N}
			\bigl[
			\e^2 ( \Delta_{S_t} \bar{\N} ) \k_+  +  (\rho_+  +  \rho_-) \nabla_{N_+} p^\star_{v,v} \bar{\N} \k_+
			\\
			& +
			\left.  
			\frac{1}{\rho_+}  \Delta_{S_t}  \N_+  p^S_{v,v}
			\bigr]
			{(- \Delta_{S_t} \bar{\N}  ) }^{2k-2} 
			\mD_{t_+} \k_+ \, dS \, 
			\right| \leq Q   \,  .
\end{align*}			
Summing the above inequality to \eqref{evolutionE_RT}, \eqref{evolutionE_2}, \eqref{evolvorticity}, and \eqref{evolvorticity-},
we see that
\begin{equation*}
\left| \frac{d}{dt} E
			-	\int _{S_t}  \bar{\N}
			\left(
			\frac{1}{\rho_+} \Delta_{S_t} \N_+ p^S_{v,v}
			\right) 
			{(-  \Delta_{S_t} \bar{\N} ) }^{2k-2} 
			\mD_{t_+} \k_+ \, dS
			\right| \leq Q   \,  .
\end{equation*}

We now define 
\begin{equation*}
K := \int _{S_t}  \bar{\N}
			\left(
			\frac{1}{\rho_+} \Delta_{S_t} \N_+ p^S_{v,v}
			\right) 
			{(-  \Delta_{S_t} \bar{\N} ) }^{2k-2} 
			\mD_{t_+} \k_+ \, dS
\end{equation*}
and focus on estimating this term.
Equation \eqref{p_vw} gives
\begin{eqnarray*}
\frac{1}{\rho_+} \mathcal{N}_+ p^S_{v,v} & = & 
		- \frac{1}{\rho_+} \mathcal{N}_+ \mathcal{N}^{-1} \left\{ 2 \nabla_{v_+^\top - v_-^\top} v_+^\bot
		- \Pi_+ ( v_+^\top, v_+^\top ) \right.
		\\
		&  &  \left. 
		- \Pi_- ( v_-^\top, v_-^\top ) 
		- \nabla_{N_+} \Delta_+^{-1}  \tr {(Dv)}^2  - \nabla_{N_-} \Delta_-^{-1}  \tr {(Dv)}^2
		\right\} \, .
\end{eqnarray*}
Using \eqref{estimateN-1} we see that the last two terms above are lower order:
\begin{align*}
{\left|  \mathcal{N}_+ \mathcal{N}^{-1} \nabla_{N_\pm} \Delta_\pm^{-1}  \tr {(Dv)}^2  \right|}_{H^{3k - \frac{1}{2}} (S_t)}
		& \leq   C  \rho_-
		{| \nabla_{N_\pm} \Delta_\pm^{-1}  \tr {(Dv)}^2 |}_{H^{3k - \frac{1}{2}} (S_t)}
		\\
		& \leq
		C \e^{3/2} \left( 1 + {| \k_+ |}_{H^{3k - \frac{3}{2}} (S_t)} \right)
		{|v|}^2_{H^{3k} (\R^n \minus S_t)}
\\  & \leq Q \, .
\end{align*}
From \eqref{N_+N^-1} and \eqref{stimaNPI} we obtain
\begin{equation*}
{ \left|  \frac{1}{\rho_+} \mathcal{N}_+ p^S_{v,v}   
					+ \frac{ \rho_- }{\rho_+  +  \rho_-} \left(  2 \nabla_{v_+^\top - v_-^\top} v_+^\bot
				    - \Pi_+ ( v_+^\top, v_+^\top )  - \Pi_- ( v_-^\top, v_-^\top ) 
				    \right)
								\right|  }_{H^{\frac{3}{2} k - \frac{1}{2} } (S_t) }  
								\! \leq   Q  \, .
\end{equation*}
Therefore, if we define 
\begin{align}
\label{K_pm^1}
K_\pm^{(1)} & :=   -  \frac{ 2 \rho_- }{\rho_+  +  \rho_-}   \int_{S_t}   
					(- \Delta_{S_t})  \nabla_{\pm v_\pm^\top} v_+^\bot
					\bar{\N} { (- \Delta_{S_t} \bar{\mathcal{N}}) }^{2k-2} \mathbf{D}_{t_+} \k_+
					\, dS \, ,
\\
\label{K_pm^2}
K_\pm^{(2)} & :=   \frac{ \rho_- }{\rho_+  +  \rho_-}   \int_{S_t}   
					(- \Delta_{S_t} )  \Pi_\pm ( v_\pm^\top, v_\pm^\top )
					\bar{\N} { (- \Delta_{S_t} \bar{\mathcal{N}}) }^{2k-2} \mathbf{D}_{t_+} \k_+ \, ,
					\, dS \, ,
\end{align}
we have $ \left| K - \left( K_+^{(1)} + K_-^{(1)} + K_+^{(2)} + K_-^{(2)} \right)  \right|  \leq Q $,
so that
\begin{equation}
\label{dtE-K}
\left| \frac{d}{dt} E  -  \left( K_+^{(1)} + K_-^{(1)} + K_+^{(2)} + K_-^{(2)} \right) \right|  \leq Q  \,  .
\end{equation}


\nl
{\it Estimate of $K^{(1)}_\pm$:} 
To deal with the tangential derivative $\nabla_{v_\pm^\top}$ 
consider flows $\Phi_\pm (\tau, \cdot)$ on $\Omega_t^+$ generated by $\H_+ v_\pm^\top$ 
and apply \eqref{commest1} to commute\footnote{
Notice that the presence of $\rho_-$ is necessary when performing this commutation since $v_-$ is involved.} 
$\mathbf{D}_\tau$ and $\Delta_{S_t}$ obtaining:
\begin{eqnarray*}
& & \left| K^{(1)}_\pm + \frac{2 \rho_-}{ \rho_+ + \rho_-} 
				\int_{S_t} \nabla_{\pm v_\pm^\top} 
				( - \Delta_{S_t} )  v_+^\bot
				\bar{\N} { (- \Delta_{S_t} \bar{\mathcal{N}}) }^{2k-2} \mathbf{D}_{t_+} \k_+
				\, dS  \right| \leq Q \, .
\end{eqnarray*}
From \eqref{D_tkappa} we have
\begin{equation*}
\rho_- { | - \Delta_{S_t} v_+^\bot  -  \mD_{t_+} \k_+  +  \nabla_{v_+^\top} \k_+| }_{H^{3k - \frac{3}{2}} (S_t)} 
= \rho_- { | v_-^\top {|\Pi|}^2| }_{H^{3k - \frac{3}{2}} (S_t)}
\leq Q
\end{equation*}
so that
\begin{eqnarray*}
& & \left| K^{(1)}_\pm  -
	\frac{2 \rho_-}{ \rho_+ + \rho_-} 
	\int_{S_t} \nabla_{\pm v_\pm^\top} \mathbf{D}_{t_+} \k_+
	\bar{\N} { (- \Delta_{S_t} \bar{\mathcal{N}}) }^{2k-2} \mathbf{D}_{t_+} \k_+
	\, dS
	\right.
	\\
	\\
	& + & \left.
	\frac{2 \rho_-}{ \rho_+ + \rho_-} 
				\int_{S_t} \nabla_{\pm v_\pm^\top} 
				\nabla_{v_+^\top} \k_+
				\bar{\N} { (- \Delta_{S_t} \bar{\mathcal{N}}) }^{2k-2} \mathbf{D}_{t_+} \k_+
				\, dS
				\right| \leq Q	\, .
\end{eqnarray*}
By the same previous commutation trick applied to the tangential derivatives, 
and the fact that $\bar{\N}$ and $-\Delta_{S_t}$ are self-adjoint, we have
\begin{align*}
& \rho_-  \,\, \left|  \int_{S_t} \nabla_{\pm v_\pm^\top} 
		\mathbf{D}_{t_+} \k_+
		\bar{\N} { (- \Delta_{S_t} \bar{\mathcal{N}}) }^{2k-2} \mathbf{D}_{t_+} \k_+
		\, dS
		\right.
		\\
		& -
 		\left.
 		\int_{S_t}  \frac{1}{2} \nabla_{\pm v_\pm^\top} \left[
		\mathbf{D}_{t_+} \k_+
			\bar{\N} { (- \Delta_{S_t} \bar{\mathcal{N}}) }^{2k-2} \mathbf{D}_{t_+} \k_+ \right]
			\, dS
			\right| \leq Q \, .
\end{align*}
We can integrate by parts the tangential derivatives in the last integral obtaining
\begin{align*}
\rho_- & \left|  \int_{S_t}  \frac{1}{2} \nabla_{v_\pm^\top} \left[
					\mathbf{D}_{t_+} \k_+
			\bar{\N} { (- \Delta_{S_t} \bar{\mathcal{N}}) }^{2k-2} \mathbf{D}_{t_+} \k_+ \right]  \, dS  \right|
			\\
			& \leq
			C \rho_- {|D v_\pm^\top|}_{L^\infty (S_t)}  {| \mathbf{D}_{t_+} \k_+ |}_{H^{3k - \frac{5}{2}} (S_t)}
			\leq Q  \, .
\end{align*}
Therefore,
\begin{eqnarray*}
& & \left| K^{(1)}_\pm  -   \frac{2 \rho_-}{ \rho_+ + \rho_-} 
			\int_{S_t} \nabla_{\pm v_\pm^\top} 
			\nabla_{v_+^\top} \k_+
			\bar{\N} { (- \Delta_{S_t} \bar{\mathcal{N}}) }^{2k-2} \mathbf{D}_{t_+} \k_+
			\, dS
			\right| \leq Q	\, .
\end{eqnarray*}
Integrating by parts and applying the usual commutation trick we can conclude
\begin{eqnarray}
\label{estimateK^1_+}
& & \left| K^{(1)}_+   + 
				  \frac{\rho_-}{ \rho_+ + \rho_-} 
					\frac{d}{dt}   \int_{S_t}
					\nabla_{v_+^\top} \k_+
				  \bar{\N} { (- \Delta_{S_t} \bar{\mathcal{N}}) }^{2k-2} \nabla_{v_+^\top}  \k_+
					\, dS
				  \right|  
				  \leq Q	\, .
\end{eqnarray}
We can handle similarly $K_-^{(1)}$
integrating again by parts, commuting the tangential derivatives, and pulling out $\mD_{t_+}$:
\begin{eqnarray}
\label{estimateK^1_-}
& & \left| K^{(1)}_-  -
				  \frac{\rho_-}{ \rho_+ + \rho_-}   \frac{d}{dt}
					\int_{S_t}
					\nabla_{v_+^\top} \k_+
				  \bar{\N} { (- \Delta_{S_t} \bar{\mathcal{N}}) }^{2k-2} \nabla_{v_-^\top}  \k_+
					\, dS
				  \right| \leq Q	\, .
\end{eqnarray}
Notice that the integrals in \eqref{estimateK^1_+} and \eqref{estimateK^1_-} 
constitute part of $E_{\ex}^{(2)}$.
The remaining contribution is going to come from the terms in \eqref{K_pm^2} involving $\Pi_\pm$.


\nl {\it Estimate of $K_\pm^{(2)}$}:
Since 
$ \D^2 \k_\pm = \nabla_{v_\pm^\top} \nabla_{v_\pm^\top} \k_\pm - \D_{v_\pm^\top} v_\pm^\top \cdot \nabla \k_\pm$,
from \eqref{formulaPIk} we get
\begin{equation*}
\rho_- \left| - \Delta_{S_t} ( \Pi_\pm (v_\pm^\top, v_\pm^\top) 
	+  \nabla_{v_\pm^\top} \nabla_{v_\pm^\top} \k_\pm  
	\right|_{H^{\frac{3}{2}k - \frac{5}{2}} (S_t) }  \leq  Q   \,  .
\end{equation*}
Therefore,
\begin{eqnarray*}
\left| K_\pm^{(2)}  +    \frac{ \rho_- }{\rho_+  +  \rho_-}   \int_{S_t}   
		\nabla_{v_\pm^\top} \nabla_{v_\pm^\top} \k_\pm
		\bar{\N} { (- \Delta_{S_t} \bar{\mathcal{N}}) }^{2k-2} \mathbf{D}_{t_+} \k_+  \, dS \right| 
		\leq  Q \, .
\end{eqnarray*}
The usual integration by parts and commutation give
\begin{align}
\label{estimateK^2_+}
& \left| K^{(2)}_+   -  \frac{\rho_-}{ 2 (\rho_+ + \rho_-) }   \frac{d}{dt}
			\int_{S_t}  \nabla_{v_+^\top} \k_+
				\bar{\N} { (- \Delta_{S_t} \bar{\mathcal{N}}) }^{2k-2} \nabla_{v_+^\top}  \k_+   \, dS
				\right| \leq Q \, ,
\\
& \left| K^{(2)}_-    +    \frac{\rho_-}{ 2 (\rho_+ + \rho_-) }   \frac{d}{dt}
			\int_{S_t}
			\nabla_{v_-^\top} \k_+
			\bar{\N} { (- \Delta_{S_t} \bar{\mathcal{N}}) }^{2k-2} \nabla_{v_-^\top}  \k_+
			\, dS
			\right| \leq Q \,\, .
\label{estimateK^2_-}
\end{align}
Gathering \eqref{estimateK^1_+}, \eqref{estimateK^1_-}, \eqref{estimateK^2_+} and \eqref{estimateK^2_-}
we have
\begin{align*}
& \left|  K_+^{(1)} + K_-^{(1)} + K_+^{(2)} + K_-^{(2)}  \right.
					\\
					& +
					\left.
					\frac{\rho_-}{ 2 (\rho_+ + \rho_-) }  \frac{d}{dt}   \int_{S_t}
					\nabla_{v_+^\top  - v_-^\top} \k_+
				  \bar{\N} { (- \Delta_{S_t} \bar{\mathcal{N}}) }^{2k-2} \nabla_{v_+^\top  -  v_-^\top}  \k_+
					\, dS
					\right| 
					\\
					& = \left| K_+^{(1)} + K_-^{(1)} + K_+^{(2)} + K_-^{(2)}  - \frac{d}{dt} E_{\ex}
					\right|   \leq Q  \, .
\end{align*}
The above estimate and \eqref{dtE-K} prove \eqref{evolutionE}  $_\Box$


\subsubsection{The Energy Inequality}
\label{secenergyinequality}
To conclude the proof of theorem \eqref{teoEE} we need to control the extra energy term $E_{\ex}$.
Integrating in time \eqref{evolutionE} gives
\begin{equation}
E (t) - E(0) - E_{\ex} (t) + E_{\ex} (0) \leq  \int_0^t Q \left( s \right) \, ds
\label{Einequality}
\end{equation}
for any $0 \leq t \leq \min \{ t_0, t_2\}$. 
Since $3k - \frac{5}{2} > \frac{n-1}{2}$, we can estimate the extra energy term \eqref{E_ex} by 
\begin{align*}
| E_{\ex} | & 
\leq   C  \rho_-  \int_{S_t} {\left| \bar{\mathcal{N}}^\frac{1}{2} {(\Delta_{S_t} \bar{\mathcal{N}} )}^{k-1} 
		\nabla_{v_+^\top  -  v_-^\top} \k_+  \right| }^2    \, dS
		\\
		& \leq
		C \rho_-  {|v (t, \cdot)|}^2_{ H^{3k - 2} (\R^n \minus S_t) }
		{|\kappa_+ (t, \cdot)|}^2_{ H^{3k - \frac{3}{2}} (S_t) }
\end{align*}
where $C$ depends only on the set $\Lambda_0$. 
Interpolating 
$\k_+$ between $H^{3k - \frac{5}{2} }$ and $H^{3k - 1}$, and using \eqref{ekappaenergy},
\eqref{venergy+}, and \eqref{venergy-}, we get
\begin{eqnarray*}
| E_{\ex} |  & \leq &  C_1 \rho_- {|v (t, \cdot)|}^2_{ H^{3k - 2} (\R^n \minus S_t) }
			{|\kappa_+ (t, \cdot)|}^{4/3}_{ H^{3k - 1} (S_t) }
			\\
			\\
			& \leq & C_1 \rho_- \e^{- 4/3} E^{2/3} 
{|v (t, \cdot)|}^2_{ H^{3k - 2} (\R^n \minus S_t) }
\end{eqnarray*}
where the constant $C_1$,
which includes ${|\k_+|}_{H^{3k - \frac{5}{2}} }$, 
depends ultimately only on the initial data and $\Lambda_0$. 
Then, we see that if $\rho_- = o(\e^{4/3})$,
as it is guaranteed by \eqref{condition},
\begin{eqnarray*}
| E_{\ex}| & \leq & 
\frac{1}{2} E  +  C_1 {|v (t, \cdot)|}^6_{ H^{3k - 2} (\R^n \minus S_t) } 
\, .
\end{eqnarray*}
%
%
%
%
%
In view of estimate \eqref{estimateD_tv} on $\mD_t v$,
we can use the Lagrangian coordinate map to get
\begin{equation*}
\left| {| v (t, \cdot) |}_{H^{3k - 2} (\R^n \minus S_t) }^6 -
				{| v (0, \cdot) |}_{H^{3k - 2} (\R^n \minus S_0) }^6 \right| \leq \int_0^t Q(s) \, ds  \,  .
\end{equation*}
Therefore,
\begin{equation*}
| E_{\ex} | \leq \frac{1}{2} E  +  C_1  \left( 1 + {|v (0,\cdot)|}^6_{ H^{3k - 2} (\R^n \minus S_0)} \right)
							+ \int_0^t Q(s) \, ds
							\leq
							\frac{1}{2} E  +  C_2 +  \int_0^t Q(s) \, ds
\end{equation*}
where $C_2$ is determined by $E_0$, the set $\Lambda_0$, and
${|v (0,\cdot)|}_{H^{3k - 2 } (\R^n \minus S_0)}$.
Inserting this last inequality in \eqref{Einequality} we finally obtain \eqref{energyestimate}.
Therefore, the energy is uniformly bounded by some constant depending only on $\Lambda_0$ and the initial data;
choosing $\mu$ in \eqref{t_0} large enough compared to the initial data 
concludes the proof of theorem \ref{teoEE}
$_\blacksquare$


\subsection{Proof of corollary \ref{corconvergence}}
The proof of strong convergence of solutions requires only some standard compactness arguments
that we are going sketch in what follows.
Let us consider any sequence of solutions of \eqref{E}-\eqref{BC} as in corollary \ref{corconvergence}
dropping the indices $m$ for convenience. Let us also denote by
$X H^l (D)$ the space
$X ( [0,T^\infty] ; H^l (D) )$ for $X = L^\infty$ or $C$
where $T^\infty$ is as in theorem \ref{teoVS1}.
Observe that the uniform bound \eqref{unifEbound} guarantees, through proposition \ref{teoVS1}, that
\begin{align}
\nonumber
& 
{| \kappa |}_{ L^\infty H^{3k-2}(S_t) } 
\hskip5pt , \hskip 7pt 
\e {| \kappa |}_{ L^\infty  H^{3k-1}(S_t) } 
\hskip5pt , \hskip 7pt 
{| v_+ |}_{ L^\infty  H^{3k} (\Omega_t^+) }  
\hskip5pt , \hskip 7pt
\\
& 
{| v_- |}_{ L^\infty  H^{3k - 1} (\Omega_t^-) }  
\hskip5pt , \hskip 7pt
\sqrt{\e} {| v_- |}_{ L^\infty  H^{3k} (\Omega_t^-) }  \leq  C_0 \, ,
\label{unifbounds}
\end{align}
for some constant $C_0$ depending only the initial data and the set $\Lambda_0$,
as in theorem \ref{teoVS1}. From now on we denote by $C_0$ any such generic constant.


Since we want to prove convergence in Lagrangian coordinates, the first step is to use \eqref{t_1}
and the uniform bounds on $v_+$ to obtain
\begin{equation*}
{|u_+|}_{L^\infty H^{3k} (\Omega_0^+)} \hskip5pt , \hskip 7pt  
{|\partial_t u_+|}_{L^\infty H^{3k} (\Omega_0^+)} \leq C_0 \, .
\end{equation*}
This shows, via the Ascoli-Arzel\'{a} theorem, that there exist a diffeomorphism 
$u^\infty \in CH^{3k}(\Omega_0^+)$
such that $u_+ \rightarrow u^\infty$ in $C^\infty H^{3k} (\Omega_0^+)$.
%
%
%
%
For the velocity field $v_+$ we immediately see from Euler equations \eqref{Ematerial},
and estimates \eqref{estimateep_k} and \eqref{estimatep_vv}, that
\begin{equation*}
{| \partial_t (v_+ \circ u_+) |}_{ L^\infty H^{3k - 3/2} (\Omega_0^+) } 
\leq C_0 {| \mD_{t_+} v_+ |}_{ L^\infty H^{3k - 3/2} (\Omega_t^+) }  \leq  C_0 \, .
\end{equation*}
Using again Ascoli-Arzel\'{a} and interpolation of Sobolev norms, 
this implies the existence of a field $v^\infty \in L^\infty H^{3k^\p} (\Omega_t^\infty)$
such that
\begin{equation}
v_+ \circ u_+ \longrightarrow v^\infty \circ u^\infty
\hskip 10pt  \mbox{in} \hskip12pt C H^{3k^\p} (\Omega_0^+)
\end{equation}
for any $k^\p < k$. It is also clear that $v^\infty$ is divergence-free on $\Omega_t^\infty := u^\infty (\Omega_0^+)$.


To prove that $v^\infty$ satisfies \eqref{E_0} pointwise, we need to obtain strong convergence of 
the time derivative of the velocity $ \partial_t (v_+ \circ u_+)$.
Using the same arguments above this reduces to check the boundedness of $\partial_t^2 (v_+ \circ u_+)$
or, equivalently, the boundedness of $\mD_{t_+} p_+$.
This can be directly obtained from the definition of $p_+$ in \eqref{physicalp},
commutator estimates in lemma \ref{lemcomm}, \eqref{estimateN-1}, and the uniform bounds \eqref{unifbounds}
which yield
\begin{equation*}
{| \partial_t^2 (v_+ \circ u_+) |}_{ L^\infty H^{3k - 3} (\Omega_0^+) } 
\leq C_0 {| \mD_{t_+} p_+ |}_{ L^\infty H^{3k - 3} (\Omega_t^+) }  \leq  C_0 \, .
\end{equation*}
The regularity of the boundary $S_t^\infty := \partial \Omega_t^\infty$ follows 
again from the same arguments since
\begin{equation*}
{|\k_+|}_{L^\infty H^{3k - 2} (S_t)} \hskip5pt , \hskip 7pt  
{|\mD_{t_+}  \k_+|}_{L^\infty H^{3k - 5/2} (S_t)} \leq C_0 \, .
\end{equation*}
Finally, again from \eqref{physicalp}, \eqref{estimateN-1}, and \eqref{unifbounds},
it is easy verify that
\begin{equation*}
{| p_- |}_{L^\infty H^{3k - 2} (S_t)}  \longrightarrow  0
\end{equation*}
so that the boundary condition \eqref{BC_0} for the pressure is also satisfied
$_\blacksquare$

\section*{Acknowledgments}
The author is indebted with Jalal Shatah and Chongchun Zeng for many enlightening discussions on the subject.
He also wishes to thank the anonymous referee for his thorough reading of the manuscript and helpful comments.

\appendix

\section{Supporting material for the proofs}
\label{tools}
In this appendix we collect some tools from \cite{shatah1,shatah2,shatah3} 
which are frequently used in our proofs. 
We first state well-know basic elliptic estimates.
The main point here is that the constants involved in these estimates are uniform over $\Lambda_0$.

%
%
\begin{lem}
\label{lemmaoperator1}
Let $\Delta^{-1}$ and $\H$ denote respectively the inverse Laplacian with Dirichlet boundary condition
and the harmonic extension operator.
Then there exists a uniform constant $C > 0$ such that for every domain $\Omega$
with $\partial \Omega := S \in \Lambda_0$
\begin{eqnarray}
\label{estimaterestriction}
& & {| {\left. f  \right |}_{S}  |}_{H^{s} (S) }  \leq 
															C {| f |}_{H^{s + \frac{1}{2}} (\Omega)}
																					\hskip8pt , \hskip8pt \for s > 0   \,  ,
\\
\label{estimateH}
& & {| \nabla \H |}_{L ( H^{s} (S), H^{s - \frac{1}{2}} (\Omega ) )} \leq C  
															\hskip8pt , \hskip8pt   \for s \in [0, 3k- 3/2] \, ;
\end{eqnarray}
moreover, for any $g$ in\footnote{
By $\dot{H}^1_0 (\Omega)$ we denote the completion of $C^\infty$ functions supported in $\Omega$
under the metric ${|\nabla g|}_{L^2 (\Omega)}$ and by ${(\dot{H}^1_0 (\Omega))}^\ast$ its dual.
} 
$H^s (\Omega) \cap {(\dot{H}_0^1 (\Omega) )}^\ast$ there exists a unique
$ q = \Delta^{-1} g$  such that
\begin{eqnarray}
\label{estimateD-1}
{| \nabla q |}_{ H^{s} (\Omega) }  \leq C \left(  {| g |}_{ H^{s-1} (\Omega) }
		+  {| g |}_{ {(\dot{H}^1_0 (\Omega))}^\ast } \right)
		\hskip8pt , \hskip8pt   \for s \in [0, 3k-1]  \, .
\end{eqnarray}
%
%
%
\end{lem}

\nl
The proof of \eqref{estimaterestriction} and \eqref{estimateH} 
is based on the construction of a suitable set of coordinates on $\Lambda_0$ 
and can be found in \cite[A.1]{shatah1}.
\eqref{estimateD-1} is just a standard elliptic estimate. 

%
%

\begin{lem}[{\bf Dirichlet-Neumann operator}]
\label{lemmaoperator2}
Let $\Omega_+ , \Omega_-$ be respectively a bounded and an unbounded domain,
such that $\R^n = \Omega_+ \cup \Omega_- \cup S$ with $S := \partial \Omega_\pm$.
The Dirichlet-Neumann operator $\N_\pm$ relative to $\Omega_\pm$ can be defined
for any $f \in H^s (S)$, $s \geq \frac{1}{2}$ and satisfies\footnote{
By $\dot{H}^s (S)$ we denote $H^s (S)$-functions with average zero.
As in \cite[Appendix A]{shatah3} we remark that since $\Omega_+$ is compact,
$\N_+$ is semipositive definite with its range being some $\dot{H}^s (S)$ space.
Therefore $\N_+^{-1}$ will always denotes the composition of the inverse of $\N_+$
with the $L^2$ orthogonal projection on functions of average zero.
Since $\Omega_-$ is unbounded, there is no restriction on $\N_-^{-1}$ for $n > 2$.
However, if $n=2$, $\N_-^{-1}$ still denotes the composition of the inverse of $\N_-$ with the
$L^2$ orthogonal projection on functions with average zero.
}
\begin{equation}
\label{estimateN0}
{| \N_\pm |}_{L ( H^{s + \frac{1}{2} } (S), H^{s - \frac{1}{2} } (S) )}  +
							{| \N_\pm^{-1} |}_{L ( \dot{H}^{s - \frac{1}{2} } (S), 
																						\dot{H}^{s + \frac{1}{2}} (S) )}
							\leq C  \,\,  ,
							\hskip6pt \for  s \in [0, 3k - 1] 
\end{equation}
for any $S \in \Lambda_0$.
In particular, if $\N$ and $\bar{\N}$ are the operators defined respectively in \eqref{N} and \eqref{barN} then
for the same $C$ as above
\begin{eqnarray}
\label{estimateN-1}
{ |  \mathcal{N}^{-1}  | }_{L ( \dot{H}^{s - \frac{1}{2} } 
															(S),  \dot{H}^{s + \frac{1}{2}} (S) )} 
							& \leq & 
							2 C \rho_-	 \,\,  ,  \hskip6pt \for  s \in [0, 3k - 1]  \,\, 
							\mbox{and} \,\, \rho_- \leq \frac{\rho_+}{2 C^2}  \, ,
\\
\nonumber
{| \bar{\N} |}_{L ( H^{s + \frac{1}{2}} (S), H^{s - \frac{1}{2}} (S) )}  
													& \leq & \frac{2 C^3 }{ \rho_+ }  
													\,\,  ,   \hskip6pt \for  s \in [0, 3k - 1]  \, ,
\\
\nonumber
{| \bar{\N}^{-1} |}_{L ( \dot{H}^{s - \frac{1}{2}} (S), \dot{H}^{s + \frac{1}{2}} (S) )}  
													& \leq & (\rho_- + \rho_+) C^3
													\,\,  ,   \hskip6pt \for  s \in [0, 3k - 1]  \, .
\end{eqnarray}
Moreover
\begin{eqnarray}
\label{N_+N^-1}
& & {\left| \N_\pm \N^{-1} - \frac{\rho_- \rho_+}{\rho_+ + \rho_-}  \right|}_{L 
													( \dot{H}^{s - \frac{1}{2}} (S), \dot{H}^{s + \frac{1}{2}} (S) )}  \leq  C \rho_-
													\hskip8pt , \hskip8pt \for  s \in [0 , 3k - 1]
\\
\nonumber
\\
\label{N_+barN}
& & {\left| (\rho_+ + \rho_-) \bar{\N}  -  \N_+ \right|}_{L 
													( H^{s + \frac{1}{2}} (S), H^{s + \frac{1}{2} } (S) )}  \leq C
													\hskip8pt , \hskip8pt \for  s \in [0, 3k - 1]
\end{eqnarray}
for some other $C$ uniform in $\Lambda_0$.
\end{lem}
\pr
The proof of \eqref{estimateN0} 
and more detailed analysis of the Dirichlet-Neumann operator can be found in \cite[A.2]{shatah1}.
Estimate \eqref{estimateN-1} is easily obtained as follows.
From the definition of $\mathcal{N}$ in \eqref{N} we have
\begin{equation*}
\mathcal{N} = \frac{\mathcal{N}_-}{\rho_-} 
\left(  \rho_- \mathcal{N}_-^{-1} \frac{\mathcal{N}_+}{\rho_+}  + I  \right) =: \frac{\mathcal{N}_-}{\rho_-} (B + I) \, .
\end{equation*}
Estimate \eqref{estimateN0} implies that for $\rho_- \leq \rho_+/(2 C^2)$, $B$
maps $H^s (S)$ to itself with norm less or equal than $C^2 \rho_- \rho_+^{-1} \leq \frac{1}{2}$.
Hence, $I + B$ is invertible and
$
\mathcal{N}^{-1} = \rho_- \sum_{j=0}^\infty {(-1)}^j B^j  \mathcal{N}_-^{-1}
$
so that
\begin{equation*}
{| \mathcal{N}^{-1} |}_{L (  \dot{H}^{s - \frac{1}{2}} (S) , \dot{H}^{s+\frac{1}{2}} (S))} 
			\leq  \rho_- C \sum_{j=0}^\infty 
			{| B |}^j_{L ( \dot{H}^{s + \frac{1}{2}} (S) , \dot{H}^{s+\frac{1}{2}} (S))} 
			\leq 2 C \rho_-   \, .
\end{equation*}
Inequalities \eqref{N_+N^-1} and \eqref{N_+barN} are a consequence of Theorem A.8 in \cite{shatah1}
where it is proved that
\begin{equation}
\label{N-Delta_S}
{ \left| \mathcal{N}_\pm  -  {(- \Delta_S )}^{\frac{1}{2}}
					\right| }_{L ( H^{s + \frac{1}{2}} (S), H^{s + \frac{1}{2}} (S) )} \leq C  
																		\hskip8pt , \hskip8pt \for  s \in [ - 3k , 3k - 2]
\end{equation}
for some $C$ uniform  in $\Lambda_0$.
To see this let us denote 
\begin{equation*}
L := {(- \Delta_S )}^{\frac{1}{2}} \N^{-1} \hskip10pt \mbox{and} 
		\hskip8pt \rho_0 := \rho_+ \rho_- / (\rho_+ + \rho_-) \ll 1 \, ;
\end{equation*}
we write 
\begin{equation*}
\mathcal{N}_\pm \N^{-1}  -  \frac{\rho_+ \rho_-}{\rho_+  + \rho_-} = 
( \mathcal{N}_\pm  -  {(- \Delta_S )}^{\frac{1}{2}} ) \N^{-1} + L - \rho_0 \, .
\end{equation*}
The first summand above satisfies the desired bound in view of \eqref{estimateN-1} and \eqref{N-Delta_S}. 
For the second summand notice that
\begin{equation*}
\rho_0 L^{-1} - I = \rho_0 \left( \frac{\N_+}{\rho_+} -  \frac{ {(- \Delta_S)}^\frac{1}{2} }{\rho_+} + 
		\frac{\N_-}{\rho_-} -  \frac{ {(- \Delta_S)}^\frac{1}{2} }{\rho_-} \right)
		{(- \Delta_S)}^{-\frac{1}{2}}
\end{equation*}
so that again by \eqref{N-Delta_S} we have
${|\rho_0 L^{-1} - I |}_{L ( \dot{H}^{s - \frac{1}{2}}, \dot{H}^{s+ \frac{1}{2}} )} \leq C $.
Therefore
\begin{equation*}
{|L - \rho_0 |}_{L ( \dot{H}^{s - \frac{1}{2}}, \dot{H}^{s+ \frac{1}{2}} )} \leq C \rho_- \, ,
\end{equation*}
and this proves \eqref{N_+N^-1}.
Finally, from the definition of $\bar{\N}$ in \eqref{barN} we have
\begin{equation*}
(\rho_+ + \rho_-) \bar{\N} - \N_+  =  \N_+ \left(  \frac{\rho_+ + \rho_-}{\rho_+ \rho_-} \N^{-1}  \N_-  -  I  \right)
\end{equation*}
so that \eqref{N_+barN} follows by \eqref{N_+N^-1} $_\Box$

%
%

In the non-linear approach to energy estimates performed in Eulerian coordinates, a key role
is played by commutators between the material derivative 
and the various differential operators appearing in the problem.

\begin{lem}[{\bf Commutator Estimates}]
\label{lemcomm}
Let $\Delta_{S_t}$, $\Delta^{-1}_\pm$ and $\H_\pm$ denote respectively the surface Laplacian on $S_t$, 
the inverse Laplacian with Dirichlet boundary conditions and the harmonic extension in the domain $\Omega_t^\pm$.
The following list of commutator estimates holds true:
\begin{align}
\label{D_tgrad}
{ \left| [ \mathbf{D}_{t_\pm}, \nabla ] \right| }_{L (H^{s} (\Omega_t^\pm), H^{s-1} (\Omega_t^\pm) ) }  
							& \leq C {|v|}_{H^{3k} (\Omega_t^\pm)}
							\hskip10pt \for 1 \leq s \leq 3k
\\
\label{D_tH}
{ \left| [ \mathbf{D}_{t_\pm}, \H_\pm ]  \right| }_{L (H^{s - \frac{1}{2}} (S_t), H^{s} (S_t) ) } 
							& \leq C {|v|}_{H^{3k} (\Omega_t^\pm)}
							\hskip10pt \for 1/2 < s \leq 3k
\\
\label{D_tDelta-1}
{ \left| \left[ \mathbf{D}_{t_\pm}, \Delta^{-1}_\pm \right] \right| }_{ 
                     L(H^{s - 2} (\Omega_t^\pm), H^{s} (\Omega_t^\pm))}
                     & \leq C {|v|}_{H^{3k} (\Omega_t^\pm)}
                     \hskip10pt \for 2 - 3 k \leq s \leq 3 k
\\
\label{commest1}
{ \left| \left[ \mathbf{D}_{t_\pm}, \mathcal{N}_\pm \right] \right| }_{
							L(H^{s} (S_t), H^{s-1} (S_t))}
							& \leq C {|v|}_{H^{3k} (\Omega_t^\pm)}
							\hskip10pt \for 1 \leq s \leq 3k - 1/2
\\
\label{commest2}
{ \left| \left[ \mathbf{D}_{t_\pm}, \Delta_{S_t} \right] \right| }_{L (H^{s} (S_t), H^{s-2} (S_t) )}
							& \leq C {|v|}_{H^{3k} (\Omega_t^\pm)}
							\hskip10pt \for 7/2 - (3/2) k < s \leq 3k - 1/2
\end{align}
with $C$ uniform for any $S_t \in \Lambda_0$.
In particular
\begin{eqnarray}
\label{D_tbarN}
{ \left| \left[ \mathbf{D}_{t_+}, \bar{\N} \right] \right| }_{
																								L(H^{s} (S_t), H^{s-1} (S_t))}
																								\leq C {|v|}_{H^{3k} (\Omega_t^+)}
																								\hskip10pt \for 1 \leq s \leq 3k - 1/2  \,  .
\end{eqnarray}
\end{lem}
\nl
Explicit formulae and estimates of the above commutators can be found in \cite[sec. 3.1]{shatah1}.

\begin{lem}[{\bf More commutators}]
\label{lemmacommp^star}
Let $p^\star_{v,v}$ and $\bar{\N}$ be defined respectively in \eqref{pwaterwave} and \eqref{barN}.
Then
\begin{align}
\label{commp^starN}
{ \left| \left[ \nabla_{N_+} p^\star_{v,v} , \bar{\N} \right]  \right| }_{
													L (H^{s} (S_t), H^{s-\frac{1}{2}} (S_t) )}
																								& \leq Q
																								\hskip10pt \for 1 \leq s \leq 3k - 3  \, 
\\
\label{commp^starDelta}
{ \left| \left[ \nabla_{N_+} p^\star_{v,v} , \Delta_{S_t} \right] \right| }_{
													L (H^{s} (S_t), H^{s-\frac{3}{2}} (S_t) )}
																								& \leq Q
																								\hskip10pt \for 3 \leq s \leq 3k - 3   \, .
\end{align}
\end{lem}

\pr
First notice that in order to prove \eqref{commp^starN} it is enough to show the bound just for $\N_+$.
It is also easy to see that $\N_+$ satisfies Leibniz' rule up to lower order terms:
\begin{equation*}
\N_+ (fg) =  g \N_+ f  + f \N_+ g - 2\nabla_{N_+} \Delta^{-1} ( \nabla \H_+ f \cdot \nabla \H_+ g) \, .
\end{equation*}
Let $a := \nabla_{N_+} p^\star_{v,v} $, then for $s < (n-1)/2$ we have\footnote{
Use the inequality
\begin{equation*}
{|f g|}_{ H^{s_1 + s_2 - \frac{n-1}{2}} (S_t) } \leq {|f|}_{ H^{s_1} (S_t) } {|g|}_{ H^{s_2} (S_t) }
\hskip10pt ( \mbox{resp.}  \hskip6pt 
{|f g|}_{ H^{s_1 + s_2 - \frac{n}{2}} (\Omega_t) } \leq {|f|}_{ H^{s_1} (\Omega_t) } {|g|}_{ H^{s_2} (\Omega_t) }  )
\end{equation*}
with $ g = \N_+ a$, $s_1 = s$ and $s_2 = n/2 - 1$
(resp. $ g = \nabla \H_+ a$, $s_1 = s - \frac{1}{2}$ and $s_2 = (n - 1)/2$ ).
}
\begin{align*}
{ \left| [ a ,\N_+] f  \right| }_{H^{s - \frac{1}{2}} (S_t) }  
							  & \leq
							  { \left| \N_+ a f \right| }_{H^{s - \frac{1}{2}} (S_t) }    
								+  2  { \left| \nabla_{N_+} 
								\Delta^{-1} ( \nabla \H_+ a \cdot \nabla \H_+ f) \right| }_{H^{s - \frac{1}{2}} (S_t) }
								\\
								& \leq
								C { \left| \N_+ a \right| }_{H^{\frac{n}{2} - 1} (S_t) } { \left| f \right| }_{H^{s} (S_t) }  
								+  C { \left| \nabla \H_+ a \cdot \nabla \H_+ f \right| }_{H^{s - 1} (\Omega_t^+) }
								\\
								& \leq
								C { \left| a \right| }_{H^{\frac{n}{2}} (S_t) } { \left| f \right| }_{H^{s} (S_t) }
								\leq Q  { \left| f \right| }_{H^{s} (S_t) }
\end{align*}
having used \eqref{estimatep^star_vv} and $3k - 1 > n/2$ in the last inequality. 
If instead $s \geq (n-1)/2$ then
\begin{align*}
{ \left| [ a ,\N_+] f  \right| }_{H^{s - \frac{1}{2}} (S_t) }  
							  & \leq  C 
							  { \left| f \right| }_{H^{s - \frac{1}{2}} (S_t) }    
							  { \left| \N_+ a \right| }_{H^{s + \frac{1}{2}} (S_t) } 
								+  
								{ \left| f \right| }_{H^{s - \frac{1}{2}} (S_t) }    
								{ \left| \nabla \H_+ a  \right| }_{H^{s + 1} (\Omega_t^+) }
								\\
								& \leq
								C { | a | }_{H^{s + \frac{3}{2}} (S_t) }  { | f | }_{ H^{s - \frac{1}{2}} (S_t) }
								\leq Q  { | f | }_{ H^{s - \frac{1}{2}} (S_t) }   \, .
\end{align*}
Similar arguments also prove \eqref{commp^starDelta} $_\Box$

\begin{lem}[{\bf Geometric Formulae}]
Let $N , \k$ and $\Pi$ denote respectively the outward unit normal, the mean-curvature
and the second fundamental form of an hypersurface $S$.
Then there exists a uniform constant $C$ such that for any $S \in \Lambda_0$
\begin{equation}
\label{stimaNPI}
{|\Pi|}_{H^s (S)} + {|N|}_{H^{s+1} (S)} \leq C (1 + {|\kappa|}_{H^s(S)}) 
																			\hskip10pt \for 3k - 5/2 \leq s \leq 3k - 1 \, .
\end{equation}
If we assume that the hypersurface $S_t$
evolves in time with velocity given by the normal component of a vector field $v$,
and let $\D$ denote the covariant derivative on $S_t$ and $\tau$ be any tangent vector,
then the following identities hold true:
\begin{align}
\label{D_tN}
\mathbf{D}_t N  & = - { \left[ (Dv)^\ast \cdot N \right]}^\top
\\
\label{D_tkappa}
\mathbf{D}_t \kappa  & =  - \Delta_{S_t} v \cdot N - 2 \Pi \cdot \left(( \left. D^\top \right|_{T\partial\Omega_t}) v \right)
\\
\nonumber
& =  - \Delta_{S_t} v^\bot - v^\bot {|\Pi|}^2 + \nabla_{v^\top} \kappa
\\
\label{D_tdS}
\mD_t dS  & = ( \mathcal{D} \cdot v^\top  + \k v^\top ) dS
\\
\label{D_tPI}
\mathbf{D}_t^\top \Pi (\tau) & =  - \D_\tau \left(  {( (Dv)^\ast N_+ )} ^\top \right)  
							- \Pi \left( {( \nabla_\tau v )}^\top \right)
\\
\label{formulaPIk}
- \Delta_{S_t} \Pi & = - \D^2 \k  +  ( {|\Pi|}^2 I - \k \Pi ) \Pi  \,  .
\end{align}
\end{lem}

\nl
The proof of the above lemma can be found in \cite{shatah1};
more specifically, identities \eqref{D_tN}, \eqref{D_tkappa} and \eqref{D_tPI} are derived in sec. 3.1,
\eqref{stimaNPI} is proved in lemma 4.7, 
and \eqref{formulaPIk} is part of the proof of proposition A.2. 


\end{document}